\documentclass[journal,onecolumn]{IEEEtran}
\usepackage{amssymb}
\usepackage{caption}
\usepackage{graphicx}
\usepackage{epsfig}
\usepackage{subcaption}
\usepackage{multicol, blindtext}
\usepackage{lipsum}
\usepackage{multirow}
\usepackage{amsmath}
\usepackage{amsfonts}
\usepackage{amssymb}
\usepackage{hyperref}
\usepackage{bm}
\usepackage[usenames]{color}
\newtheorem{example}{Example}
\newtheorem{definition}{Definition}
\newtheorem{proposition}{Proposition}
\newtheorem{lemma}{Lemma}
\newtheorem{theorem}{Theorem}
\newtheorem{corollary}{Corollary}
\usepackage{algorithm,algorithmic}
\usepackage{float}
\floatstyle{boxed}
\restylefloat{figure}

\newtheorem{remark}{Remark}
\newcommand{\oomit}[1]{}

\begin{document}

\title{Reach-Avoid Differential Games Based on Invariant Generation}

\author{Bai Xue$^1$ and Qiuye Wang$^1$ and Naijun Zhan$^1$ and Martin Fr\"anzle$^2$  and Shenghua Feng$^1$\thanks{1. State Key Lab. of Computer Science, Institute of Software, CAS, China.(\{xuebai,wangqye,znj,fengsh\}@ios.ac.cn).}
\thanks{2. Carl von Ossietzky Universit\"at Oldenburg, Germany.(martin.fraenzle@uni-oldenburg.de)}
}

\maketitle

\begin{abstract}
Reach-avoid differential games play an important role in collision avoidance, motion planning and control of aircrafts, and related applications. The central problem is the computation of the set of initial states from which the ego player can enforce the satisfiability of  safety specifications over a specified time horizon. Previous methods addressing this problem mostly focus on finite time horizons. We study this problem in the context of the infinite time horizon, where the ego player aims to perpetually force the system to satisfy certain safety specification while the mutual other player attempts to enforce a violation of this safety specification. The problem is studied within the Hamilton-Jacobi reachability framework with unique Lipschitz continuous viscosity solutions. The continuity and uniqueness property of the viscosity solution facilitates the use of contemporary numerical methods to solve this problem with an appropriate number of state variables. An example adopted from a Moore-Greitzer jet-engine model is employed to illustrate our approach.
\end{abstract}
\begin{keywords}
Differential Games; Hamilton-Jacobi Equation; Invariant Sets
\end{keywords}

\section{Introduction}
\label{intro}
Differential games, i.e.\ dynamic games featuring an evolution governed by differential equations, have many important applications in engineering domains, e.g., in the analysis of collision avoidance \cite{mylvaganam2017,xue2017reach}, energy management \cite{dockner2000} and safe reinforcement learning \cite{sharma2010}. They model a form of strategic interactions among rational players, where each player makes decisions in light of its own preference while expecting adversarial actions from the mutual other player. As the resulting winning strategies are robust against any possible action of the adversary, differential games  have in recent years received growing interest as a model facilitating synthesis of reliable control strategies for safety-critical systems.

Differential games were initiated by Rufus Isaacs in the early 1950s when he studied military pursuit-evasion problems while working in the Rand Corporation. The pursuit-evasion game he studied is a two-player zero-sum game, where the players have completely opposite interests \cite{isaacs1999}. A challenging class of differential games is known as reach-avoid games, which are  to determine the set of states from which the ego player is able to drive the system to reach a desired target set of states while staying away from an avoid set, regardless of the opposing actions of the mutual other player - this set goes by many names in the literature, e.g.,  discriminating kernels \cite{aubin1991}, backward reachable sets \cite{mitchell2005} and stable bridges \cite{subbotin2013}. The present work studies this problem within the Hamilton-Jacobi reachability framework. Hamilton-Jacobi reachability analysis addresses reachability problem by exploiting the link to optimal control through viscosity solutions of Hamilton-Jacobi equations \cite{bansal2017}. It extends the use of Hamilton-Jacobi equations, which are widely used in optimal control theory \cite{bardi1999}, to perform reachability analysis over both finite time horizons \cite{lygeros2004,mitchell2005,margellos2011,altarovici2013,fisac2015} and the infinite time horizon \cite{camilli2001,grune2011,grune2015}. While computationally intensive, Hamilton-Jacobi reachability approaches are still appealing nowadays due to the availability of modern numerical tools such as \cite{mitchell2007,bokanowski2013,falcone2016}, which allow solving associated game problems conveniently with appropriate number of state variables. Within the Hamilton-Jacobi framework, continuity of viscosity solutions is a desirable property from a theoretical point of view since discontinuities may invalidate uniqueness of the solution \cite{Bardi1997,fialho1999}. Continuity is also desirable from a numeric computation point of view, since rigorous convergence results for numerical approximations to the derived Hamilton-Jacobi equation usually require continuity of the solution. Unfortunately, reachability analysis under state constraints may induce discontinuities in the viscosity solutions, see for instance \cite{koike1995,Bardi1997,fialho1999,bardi2000,CardaliaguetQS00,quincampoix2002}, unless the dynamics satisfy special assumptions at the boundary of state constraints, e.g, inward pointing qualification assumption \cite{soner1986a,soner1986}, outward pointing condition \cite{frankowska2000} and vanishing on the boundary \cite{bardi2000}. These conditions are, however, restrictive and viscosity solution can therefore be discontinuous in general. \oomit{such assumption is not satisfied and the viscosity solution could be discontinuous. In this framework, \cite{frankowska2000} introduced another controllability assumption, i.e. outward pointing condition. Under this assumption it is still possible to characterize the value function as the unique lower semi-continuous solution of an HJI equation. However, discontinuity hampers the application of existing numerical methods.} Recently, without requiring such assumptions, \cite{bokanowski2010} infers a modified Hamilton-Jacobi equation and considers reachability problems over finite time horizons for state-constrained problems with control inputs. The modified Hamilton-Jacobi equation exhibits a unqiue continuous viscosity solution. Based on such Hamilton-Jacobi formulation in \cite{bokanowski2010}, \cite{margellos2011} studies the finite-time reach-avoid games for state-constrained systems. \cite{fisac2015} further  investigates differential games over finite time horizons where the  target  set,  the state constraint set, and dynamics  are allowed  to  be  time-varying. Recently, \cite{grune2011} considers the generation of the region of attraction over the infinite time horizon. The region of attraction here is the set of initial states that are controllable in that they can be driven, using an admissible control while respecting a set of state constraints, to asymptotically approach an equilibrium state. \cite{xue2019} studies the problem of computing robust invariant sets over the infinite time horizon for state-constrained perturbed nonlinear systems without control inputs within the Hamilton-Jacobi reachability framework, where a robust invariant set is a set of states such that every possible trajectory starting from it never violates the given state constraint, irrespective of the actual perturbation. In \cite{xue2019} the maximal robust invariant set is described as the zero level set of the unique Lipschitz-continuous viscosity solution to a Hamilton-Jacobi-Bellman (HJB) equation. However, to the best of our knowledge there is no previous work on the use of Hamilton-Jacobi equations having continuous viscosity solutions to address the infinite time reach-avoid differential game for state-constrained systems.

In this paper we therefore extend the Hamilton-Jacobi formulation from \cite{xue2019} to address infinite time reach-avoid differential games for state-constrained systems. In the reach-avoid game, we consider computation of the lower robust controlled invariant set and the upper robust controlled invariant set. The lower robust controlled invariant set is a set of initial states such that  for any finite time horizon, there exists a nonanticipative strategy for the ego player which makes the system satisfy the specified safety specification, irrespective of actions of the mutual other player. The upper robust controlled invariant set is a set of initial states such that for any nonanticipative strategies of the mutual other player and any finite time horizon, there exists a action for the ego player which makes the system satisfy the specified safety specification. We characterize the lower robust controlled invariant set as the zero level set of a unique bounded Lipschitz continuous viscosity solution to a Hamilton-Jacobi equation with sup-inf Hamiltonian and the upper robust controlled invariant set as the zero level set of a unique bounded Lipschitz continuous viscosity solution to a Hamilton-Jacobi equation with inf-sup Hamiltonian, respectively. Under the classical Isaacs condition, these two sets coincide. The continuity of viscosity solutions facilitates the use of existing numerical methods to solve the associated Hamilton-Jacobi equations. An example adopted from modern Moore-Greitzer jet engine model \cite{sassi2012} is employed to demonstrate our approach.

This paper is structured as follows: Section \ref{Pre} gives a detailed introduction of the differential game of interest in this paper, including the notion of lower and upper robust controlled invariant sets. Section \ref{HJB} formulates the computation of both lower and upper robust controlled invariant sets within the framework of Hamilton-Jacobi type partial differential equation. After demonstrating our approach on one example in Section \ref{examples}, we conclude this paper in Section \ref{con}.

\section{Differential Game Formulation}
\label{Pre}
In this section we introduce the definitions and notations which are employed in the rest of this paper.  The following basic notations will be used in what follows: $\mathbb{R}^n$ denotes
the set of n-dimensional real vectors. $\|\bm{x}\|$ denotes the 2-norm, i.e., $\|\bm{x}\|:=\sqrt{\sum_{i=1}^n x_i^2}$, where $\bm{x}=(x_1,\ldots,x_n)$. $C^{\infty}(\mathbb{R}^n)$ denotes the set of smooth functions over $\mathbb{R}^n$. Vectors are denoted by boldface letters.

We consider a reach-avoid differential game with dynamics given by
\begin{equation}
\label{system}
\begin{cases}
&\dot{\bm{x}}(s)=\bm{f}(\bm{x}(s),\bm{u}(s),\bm{d}(s))\\
&\bm{x}(0)=\bm{x}_0\in \mathcal{X}.
\end{cases}
\end{equation}
Here we assume that $\bm{f}(\bm{x},\bm{u},\bm{d}): \mathbb{R}^n\times U\times D\mapsto \mathbb{R}^n$ is continuous over $\bm{x}$, $\bm{u}$ and $\bm{d}$, and locally Lipschitz in $\bm{x}$ uniformly in $\bm{u}$ and $\bm{d}$. The sets $\mathcal{X}$, $U$ and $D$ are compact subsets of finite dimensional spaces $\mathbb{R}^n$, $\mathbb{R}^m$ and $\mathbb{R}^l$ respectively, and the controls $\bm{u}(\cdot):[0,\infty)\mapsto U$ and $\bm{d}(\cdot):[0,\infty)\mapsto D$ are measurable functions. We define
\begin{equation*}
\begin{split}
&\mathcal{U}=\{\bm{u}(\cdot):[0,\infty)\mapsto U, \text{measurable}\} \text{~and}\\
&\mathcal{D}=\{\bm{d}(\cdot):[0,\infty)\mapsto D, \text{measurable}\}
\end{split}
\end{equation*}
as the respective sets of control functions.

 As point-wise limits of measurable functions are measurable, $\mathcal{U}$ is a closed subset, and consequently compact in the topology of point-wise convergence \cite{platzer2017}. Analogously, $\mathcal{D}$ is also compact in the topology of point-wise convergence. Throughout this paper we will investigate the situation in which the ego player wants to control the system to stay within a set  while the mutual other player attempts to prevent this. For this reason, we will usually interpret $\bm{u}(\cdot)$ as a control action while considering $\bm{d}(\cdot)$ as an adversarial perturbation. The trajectory  of system \eqref{system} under the control of $\bm{u}(\cdot)\in \mathcal{U}$ and $\bm{d}(\cdot)\in \mathcal{D}$ is denoted by $\bm{\phi}_{\bm{x}_0}^{\bm{u},\bm{d}}(\cdot): \mathbb{R}\mapsto \mathbb{R}^n$ with $\bm{\phi}_{\bm{x}_0}^{\bm{u},\bm{d}}(0)=\bm{x}_0$. The game is investigated in the framework of non-anticipative strategy, whose concept is formally presented in Definition \ref{nonanticipative}.
\begin{definition}
\label{nonanticipative}
We say that a map $\bm{\alpha}(\cdot): \mathcal{D}\mapsto \mathcal{U}$ is a non-anticipative strategy (for the ego player) if it satisfies the following condition:

For $\bm{d}_1(\cdot)$, $\bm{d}_2(\cdot)\in \mathcal{D}$ with $\bm{d}_1(t)=\bm{d}_2(t)$ almost everywhere on $t\in [0,s]$ for any $s\geq 0$, $\bm{\alpha}(\bm{d}_1)(t)$ and $\bm{\alpha}(\bm{d}_2)(t)$ coincide almost everywhere on $[0,s]$.
The set of non-anticipative strategies $\bm{\alpha}(\cdot)$ for the ego player is denoted by $\Gamma$.

Non-anticipative strategies for the other player $\bm{\beta}(\cdot):\mathcal{U}\mapsto \mathcal{D}$ are defined similarly. Its corresponding set is denoted by $\Delta$.
\end{definition}

According to Remark 5.9 in \cite{platzer2017}, $\Gamma$ and $\Delta$ are compact in the product topology of point-wise convergence. Based on the non-anticipative strategies in Definition \ref{nonanticipative}, we define two types of robust controlled invariant sets, i.e., lower robust controlled invariant set and upper robust controlled invariant set.

 \begin{definition}
 \label{ICC}
 Let $\mathcal{X}_{\epsilon}=\{\bm{x}\in \mathbb{R}^n\mid h(\bm{x})\leq \epsilon\}$ be a set in $\mathbb{R}^n$, where $h(\bm{x})$ is a bounded and locally Lipschitz continuous function in $\mathbb{R}^n$,

1)  The lower robust controlled invariant set $\mathcal{R}^{-}$ of system \eqref{system} is the set of states $\bm{x}$'s such that  for any $ \epsilon >0$ and any $T\geq 0$, there exists a non-anticipative strategy $\bm{\alpha}(\cdot)\in \Gamma$ such that for any perturbation $\bm{d}(\cdot)\in \mathcal{D}$ the corresponding trajectory $\bm{\phi}_{\bm{x}}^{\bm{\alpha}(\bm{d}),\bm{d}}(t)$ stays inside $\mathcal{X}_{\epsilon}$ for $t\in [0,T]$,
 i.e.,
\begin{equation*}
\begin{split}
 \mathcal{R}^{-}=\{\bm{x}\in \mathbb{R}^n\mid &\forall \epsilon>0, \forall T\geq 0, \exists \bm{\alpha}(\cdot) \in \Gamma, \forall \bm{d}(\cdot) \in \mathcal{D}, \forall t\in [0,T], \bm{\phi}_{\bm{x}}^{\bm{\alpha}(\bm{d}),\bm{d}}(t)\in \mathcal{X}_{\epsilon}\}.
 \end{split}
 \end{equation*} 
 2). The upper robust controlled invariant set $\mathcal{R}^+$ of system \eqref{system} is the set of states $\bm{x}$'s such that for any $T\geq 0$ and  any $\epsilon>0$ and any non-anticipative strategies $\bm{\beta}(\cdot)\in \Delta$, there exists a control $\bm{u}(\cdot)\in \mathcal{U}$ such that the trajectory $\bm{\phi}_{\bm{x}}^{\bm{u},\bm{\beta}(\bm{u})}(t)$ stays inside $\mathcal{X}_{\epsilon}$ for $t\in  [0,T]$, i.e.,
\begin{equation*}
\begin{split}
\mathcal{R}^{+}=\{\bm{x}\in \mathbb{R}^n\mid & \forall \epsilon>0, \forall T\geq 0, \forall \bm{\beta}(\cdot) \in \Delta, \exists\bm{u}(\cdot) \in \mathcal{U}, \forall t\in [0,T], \bm{\phi}_{\bm{x}}^{\bm{u},\bm{\beta}(\bm{u})}(t)\in \mathcal{X}_{\epsilon}\}.
 \end{split}
\end{equation*} 
 \end{definition}

Note that the assumption on the boundedness of $h(\bm{x})$ over $\bm{x}\in \mathbb{R}^n$ is not strict since if $h(\bm{x})$ is unbounded, then $h(\bm{x}):=\frac{h(\bm{x})}{1+h^2(\bm{x})}$is bounded and $\mathcal{X}$ is still equal to $\{\bm{x}\in \mathbb{R}^n\mid h(\bm{x})\leq 0\}$.

An immediate conclusion from Definition \ref{ICC} is presented in Corollary \ref{in}.
\begin{corollary}
\label{in}
$\mathcal{R}^-\subseteq \mathcal{X}$ and $\mathcal{R}^{+}\subseteq \mathcal{X}$.
\end{corollary} 
\begin{proof}
Let $\bm{x}\in \mathcal{R}^-$ but $\bm{x}\notin \mathcal{X}$. Obviously, there exists $\epsilon_1>0$ such that $h(\bm{x})=\epsilon_1$. Therefore, 
\[\exists \epsilon<\epsilon_1, \exists T=0, \forall \bm{\alpha}(\cdot) \in \Gamma, \exists \bm{d}(\cdot) \in \mathcal{D}, \exists t\in [0,T], \bm{\phi}_{\bm{x}}^{\bm{\alpha}(\bm{d}),\bm{d}}(t)\notin \mathcal{X}_{\epsilon},\] contradicting $\bm{x}\in \mathcal{R}^-$. Consequently, $\bm{x}\in \mathcal{X}$ and thus $\mathcal{R}^-\subseteq \mathcal{X}$. 

Analogously, we have $\mathcal{R}^{+}\subseteq \mathcal{X}$.
\end{proof}

\section{Characterization of $\mathcal{R}^{\pm}$ Using HJI}
\label{HJB}
In this section we characterize the lower and upper robust controlled invariant sets $\mathcal{R}^{-}$ and $\mathcal{R}^{+}$ using Hamilton-Jacobi equations with  sup-inf and inf-sup Hamiltonians respectively.

In order to obtain Hamilton-Jacobi equations for characterizing these two robust controlled invariant sets $\mathcal{R}^{-}$ and $\mathcal{R}^{+}$, for any solution $\bm{\phi}_{\bm{x}}^{\bm{u},\bm{d}}(\cdot)$ of \eqref{system} with initial value $\bm{x}$ we associate a payoff which depends on $\bm{u}(\cdot)\in \mathcal{U}$ and $\bm{d}(\cdot)\in \mathcal{D}$ and is denoted by
\begin{equation}
\label{J}
J(\bm{x},\bm{u},\bm{d}):=\sup_{t\in[0,\infty)}e^{-\gamma t}h(\bm{\phi}_{\bm{x}}^{\bm{u},\bm{d}}(t)),
\end{equation}
where $\gamma$ is a scalar constant valued in $(0,\infty)$.

\begin{remark}
Note that we only assume that $\bm{f}(\bm{x},\bm{u},\bm{d})$ in system \eqref{system} is locally Lipschitz continuous over $\bm{x}$ uniformly in $\bm{u}\in U$ and $\bm{d}\in D$, this generally can not guarantee the global existence of the Caratheodory solution $\bm{\phi}_{\bm{x}}^{\bm{u},\bm{d}}(t)$ over $t\in [0,\infty)$ for every $\bm{x}\in \mathbb{R}^n$.
Thanks to  Kirszbraun's extension theorem for Lipschitz maps \cite{valentine1945a}, we can construct a global Lipschitz function $\bm{F}(\bm{x},\bm{u},\bm{d})$ such that $\bm{F}(\bm{x},\bm{u},\bm{d})=\bm{f}(\bm{x},\bm{u},\bm{d})$ over $\bm{x}\in B$ and its global Lipschitz constant $L_{\bm{F}}$ is equal to $L_{\bm{f}}$, where $L_{\bm{f}}$ is the Lipschitz constant of the function $\bm{f}(\bm{x},\bm{u},\bm{d})$ over $B$ and $\mathcal{X}\subset B$. For instance, $\bm{F}(\bm{x},\bm{d}):=\inf_{\bm{y}\in B}(\bm{f}(\bm{y},\bm{u},\bm{d})+A L_{\bm{f}}\|\bm{x}-\bm{y}\|)$ satisfies such requirement, where $A$ is an $n-$dimensional vector with each component equaling to one. Since $\bm{F}(\bm{x},\bm{u},\bm{d})=\bm{f}(\bm{x},\bm{u},\bm{d})$ over $\bm{x}\in \mathcal{X}$, the dynamics of the system \eqref{system} and the system $\dot{\bm{x}}=\bm{F}(\bm{x},\bm{u},\bm{d})$ are the same within the set $\mathcal{X}$. From Corollary \ref{in}, we have that the sets $\mathcal{R}^-$ and $\mathcal{R}^+$ in Definition \ref{ICC} under the system $\dot{\bm{x}}=\bm{F}(\bm{x},\bm{u},\bm{d})$ remain the same. Also,  the original system \eqref{system} is sufficient for existing numerical methods to compute $\mathcal{R}^-$ and $\mathcal{R}^+$ on the set $B$ since  $\bm{F}(\bm{x},\bm{u},\bm{d})=\bm{f}(\bm{x},\bm{u},\bm{d})$ over $\bm{x}\in B$. Therefore, for ease exposition we still use the original system \eqref{system} for theoretical analysis in the remainder of this paper with assumed global existence of a unique solution for each $\bm{x}\in \mathbb{R}^n$. In the sequel we continue exploring properties of the function $J(\bm{x},\bm{u},\bm{d})$ in \eqref{J}.
\end{remark}

\begin{lemma}
\label{JJ}
$J(\bm{x},\bm{u},\bm{d})$ in \eqref{J} is continuous over $(\bm{u}(\cdot),\bm{d}(\cdot)) \in \mathcal{U}\times \mathcal{D}$.
\end{lemma}
\begin{proof}
Assume that $\lim_{n\rightarrow \infty}\bm{u}_n(t)=\bm{u}(t)$ and $\lim_{n\rightarrow \infty}\bm{d}_n(t)=\bm{d}(t)$ point-wise, where $\bm{u}_n(\cdot)\in \mathcal{U}$ and $\bm{d}_n(\cdot) \in \mathcal{D}$ for $n\geq 1$, we will prove that for every $\epsilon>0$, there exists $N>0$ such that 
\begin{equation*}
|J(\bm{x},\bm{u},\bm{d})-J(\bm{x},\bm{u}_n,\bm{d}_n)|<\epsilon, \forall n>N.
\end{equation*}

Since $h(\bm{x})$ is bounded over $\mathbb{R}^n$, there exists $M\in [0,\infty)$ such that $\|h(\bm{x})\|\leq M$ over $\mathbb{R}^n$. Consequently, we have that for given $\epsilon>0$, there exists $T>0$ such that 
\begin{equation*}
|e^{-\gamma t}h(\bm{\phi}_{\bm{x}}^{\bm{u},\bm{d}}(t))-e^{-\gamma t}h(\bm{\phi}_{\bm{x}}^{\bm{u}_n,\bm{d}_n}(t))|\leq 2Me^{-\gamma t}<\frac{\epsilon}{2},\forall \bm{u}_n(\cdot)\in \mathcal{U}, \forall \bm{d}_n(\cdot)\in \mathcal{D}, \forall t\geq T
\end{equation*}
holds. Therefore,
\begin{equation*}
\begin{split}
&|J(\bm{x},\bm{u},\bm{d})-J(\bm{x},\bm{u}_n,\bm{d}_n)|\\
&\leq \sup_{t\in [0,\infty)}|e^{-\gamma t}h(\bm{\phi}_{\bm{x}}^{\bm{u},\bm{d}}(t))-e^{-\gamma t}h(\bm{\phi}_{\bm{x}}^{\bm{u}_n,\bm{d}_n}(t))|\\
&\leq \sup_{t\in [0,T]}|e^{-\gamma t}h(\bm{\phi}_{\bm{x}}^{\bm{u},\bm{d}}(t))-e^{-\gamma t}h(\bm{\phi}_{\bm{x}}^{\bm{u}_n,\bm{d}_n}(t))|+\frac{\epsilon}{2}, \forall n\geq 1.\\
\end{split}
\end{equation*}

From Lemma 5.8 in \cite{platzer2017} stating that if $\lim_{n\rightarrow \infty}\bm{u}_n(t) =\bm{u}(t)$ and $\lim_{n\rightarrow \infty}\bm{d}_n(t)=\bm{d}(t)$ point-wise, then $\lim_{n\rightarrow \infty}\bm{\phi}_{\bm{x}}^{\bm{u}_n,\bm{d}_n}(t)= \bm{\phi}_{\bm{x}}^{\bm{u},\bm{d}}(t)$ uniformly on $[0,T]$, we finally have that for given $\epsilon>0$, there exists $N>0$ such that
\begin{equation*}
|J(\bm{x},\bm{u},\bm{d})-J(\bm{x},\bm{u}_n,\bm{d}_n)|<\epsilon, \forall n\geq N.
\end{equation*}
\end{proof}

For the payoff $J(\bm{x},\bm{u},\bm{d})$, we respectively define the lower value function $V^-$ and upper value function $V^+$ as follows:
\begin{equation}
\label{vminus}
V^{-}(\bm{x}):=\inf_{\bm{\alpha}(\cdot) \in \Gamma}\sup_{\bm{d}(\cdot)\in \mathcal{D}}J(\bm{x},\bm{\alpha}(\bm{d}),\bm{d})~\text{and}
\end{equation}
\begin{equation}
\label{vplus}
V^{+}(\bm{x}):=\sup_{\bm{\beta}(\cdot)\in \Delta}\inf_{\bm{u}(\cdot)\in \mathcal{U}}J(\bm{x},\bm{u},\bm{\beta}(\bm{u})).
\end{equation}
We show that the zero level sets of the lower value function $V^-$  and the upper value function $V^+$  are respectively the lower robust controlled invariant set $\mathcal{R}^-$ and the upper robust controlled invariant set $\mathcal{R}^+$, i.e. $\mathcal{R}^-=\{\bm{x}\in \mathbb{R}^n\mid V^{-}(\bm{x})=0\}$ and $\mathcal{R}^+=\{\bm{x}\in \mathbb{R}^n\mid V^{+}(\bm{x})=0\}$. Before justifying this statement, we need an intermediate proposition stating that both the lower value function $V^{-}$ and the upper value function $V^+$ are positive and bounded over $\mathbb{R}^n$. 
\begin{proposition}
\label{bounded}
$V^{-}(\bm{x})$ is non-negative and bounded over $\bm{x}\in \mathbb{R}^n$. Analogously, $V^+(\bm{x})$ is non-negative and bounded over $\bm{x}\in \mathbb{R}^n$ as well.
\end{proposition}
\begin{proof}
We just prove the statement pertinent to $V^-(\bm{x})$. The similar proof procedure applies to $V^+$ as well.

Since $h(\bm{x})$ is bounded over $\mathbb{R}^n$, we have that $$\lim_{t\rightarrow \infty}e^{-\gamma t}h(\bm{\phi}_{\bm{x}}^{\bm{\alpha}(\bm{d}),\bm{d}}(t))=0, \forall \bm{\alpha}(\cdot)\in \Gamma,
\forall \bm{d}(\cdot)\in \mathcal{D}, \forall \bm{x}\in \mathbb{R}^n.$$
This implies that $J(\bm{x},\bm{\alpha}(\bm{d}),\bm{d})\geq 0, \forall \bm{\alpha}(\cdot)\in \Gamma, \forall \bm{d}(\cdot)\in \mathcal{D},\forall \bm{x}\in \mathbb{R}^n.$
Thus, \[\sup_{\bm{d}(\cdot)\in \mathcal{D}}J(\bm{x},\bm{\alpha}(\bm{d}),\bm{d})\geq 0, \forall \bm{\alpha}(\cdot)\in \Gamma, \forall \bm{x}\in \mathbb{R}^n.\]
 Consequently, $V^{-}(\bm{x})\geq 0$ for $\bm{x}\in \mathbb{R}^n$.

The boundedness of $V^-$ is guaranteed by the fact that $$J(\bm{x},\bm{\alpha}(\bm{d}),\bm{d})\leq M, \forall \bm{\alpha}(\cdot)\in \Gamma, \forall \bm{d}(\cdot)\in \mathcal{D}, \forall \bm{x}\in \mathbb{R}^n,$$ where $M$ is a positive value such that
$|h(\bm{x})|\leq M$ over $\bm{x}\in \mathbb{R}^n$. Thus,
$V^-(\bm{x})\leq M$ over $\bm{x}\in \mathbb{R}^n.$
\end{proof}

\begin{lemma}
\label{inclusion}
$\mathcal{R}^{-}=\{\bm{x}\mid V^{-}(\bm{x})= 0\}$ and
$\mathcal{R}^{+}= \{\bm{x}\mid V^{+}(\bm{x})=0\}$.
\end{lemma}
\begin{proof}
1. For the statement $\mathcal{R}^{-}=\{\bm{x}\mid V^{-}(\bm{x})= 0\}$, we first prove $\mathcal{R}^{-}\subseteq\{\bm{x}\mid V^{-}(\bm{x})= 0\}$.

Consider $\bm{x}\in \mathcal{R}^{-}$.  It implies $\sup_{\bm{d}(\cdot)\in \mathcal{D}}\sup_{t\in [0,\infty)}h(\bm{\phi}_{\bm{x}}^{\bm{\alpha}(\bm{d}),\bm{d}}(t))\leq \epsilon$ and consequently $\sup_{\bm{d}(\cdot)\in \mathcal{D}}\sup_{t\in [0,\infty)}e^{-\gamma t}h(\bm{\phi}_{\bm{x}}^{\bm{\alpha}(\bm{d}),\bm{d}}(t))\leq \epsilon.$ Thus,
\begin{equation}
\begin{split}
V^{-}(\bm{x})&=\inf_{\bm{\alpha}(\cdot) \in \Gamma} \sup_{\bm{d}(\cdot)\in \mathcal{D}}J(\bm{x},\bm{\alpha}(\bm{d}),\bm{d})\leq \sup_{\bm{d}(\cdot)\in \mathcal{D}}J(\bm{x},\bm{\alpha}(\bm{d}),\bm{d})\leq \epsilon.
\end{split}
\end{equation}
Since $\epsilon$ is an arbitrary positive number, $V^-(\bm{x})\leq 0$. In addition, according to Proposition \ref{bounded} which states that $V^-(\bm{x})\geq 0$ over $\mathbb{R}^n$, we conclude that $\mathcal{R}^{-}\subseteq\{\bm{x}\in \mathbb{R}^n\mid V^{-}(\bm{x})=0\}.$

Next, we prove that $\{\bm{x}\in \mathbb{R}^n\mid V^{-}(\bm{x})=0\}\subseteq \mathcal{R}^{-}.$ 

Assume that $\bm{x}_0\in \{\bm{x}\in \mathbb{R}^n\mid V^{-}(\bm{x})=0\}$ but $\bm{x}_0\notin \mathcal{R}^-$.  Therefore, we have 
\[\exists \epsilon>0, \exists T\geq 0, \forall \bm{\alpha}(\cdot)\in \Gamma, \exists \bm{d}(\cdot)\in \mathcal{D}, \exists t\in [0,T], \bm{\phi}_{\bm{x}_0}^{\bm{\alpha}(\bm{d}),\bm{d}}(t)\notin \mathcal{X}_{\epsilon}.\]
Therefore, $\sup_{t\in [0,\infty)}e^{-\gamma t} h(\bm{\phi}_{\bm{x}_0}^{\bm{\alpha}(\bm{d}),\bm{d}}(t))\geq e^{-\gamma T} \epsilon$ for $\bm{\alpha}(\cdot)\in \Gamma$ and consequently 
\[\inf_{\bm{\alpha(\cdot)}\in \Gamma}\sup_{\bm{d}(\cdot)\in \mathcal{D}}\sup_{t\in [0,\infty)}e^{-\gamma t} h(\bm{\phi}_{\bm{x}_0}^{\bm{\alpha}(\bm{d}),\bm{d}}(t))\geq e^{-\gamma T} \epsilon,\] contradicting $V^{-}(\bm{x}_0)=0$. Thus, $\bm{x}_0\in \mathcal{R}^-$ and further $\{\bm{x}\mid V^-(\bm{x}_0)=0\} \subseteq \mathcal{R}^-$.

In summary, we have $\mathcal{R}^{-}=\{\bm{x}\in \mathbb{R}^n\mid V^{-}(\bm{x})=0\}$.

2. We prove the second statement that $\mathcal{R}^{+}\subseteq \{\bm{x}\mid V^+(\bm{x})=0\}$. Let $\bm{x}\in \mathcal{R}^{+}$ and $V^+(\bm{x})=\delta>0$. We will derive a contradiction. Due to $V^+(\bm{x})=\delta>0$, there exists $\bm{\beta}_1(\cdot)\in\Delta$ such that $\inf_{\bm{u}(\cdot)\in \mathcal{U}}J(\bm{x},\bm{u},\bm{\beta}_1(\bm{u}))>\frac{\delta}{2}$, implying that $J(\bm{x},\bm{u},\bm{\beta}_1(\bm{u}))>\frac{\delta}{2}$ for all $\bm{u}(\cdot)\in \mathcal{U}$. Due to the fact that there exists $T'>0$ such that 
\begin{equation*} e^{-\gamma T'}h(\bm{\phi}_{\bm{x}}^{\bm{u},\bm{\beta}(\bm{u})}(T')) \leq   e^{-\gamma T'}M\leq \frac{\delta}{2}, \forall \bm{u}(\cdot)\in \mathcal{U}, \forall \bm{\beta}(\cdot)\in \Delta,
\end{equation*}
where $M$ is a positive value such that $|h(\bm{x})|\leq M$ over $\bm{x}\in \mathbb{R}^n$. 
 there exists $T_{\bm{u}}\in [0,T']$  for $\bm{u}(\cdot) \in \mathcal{U}$ such that 
\begin{equation*}
e^{-\gamma T_{\bm{u}}}h(\bm{\phi}_{\bm{x}}^{\bm{u},\bm{\beta}_1(\bm{u})}(T_{\bm{u}}))>\frac{\delta}{2}
\end{equation*}and therefore, $\bm{\phi}_{\bm{x}}^{\bm{u},\bm{\beta}_1(\bm{u})}(T_{\bm{u}})\notin \mathcal{X}_{\frac{\delta}{2}}$ for $\bm{u}(\cdot) \in \mathcal{U}$, contradicting $\bm{x}\in \mathcal{R}^{+}$. $\mathcal{R}^{+}\subseteq\{\bm{x}\mid V^{+}(\bm{x})\leq 0\}$ holds. In addition, according to Proposition \ref{bounded} which states that $V^+(\bm{x})\geq 0$ for $\bm{x}\in \mathbb{R}^n$, we have $\mathcal{R}^{+}\subseteq\{\bm{x}\in \mathbb{R}^n\mid V^+(\bm{x})=0\}.$

Next, we show that $\{\bm{x}\in \mathbb{R}^n\mid V^+(\bm{x})=0\}\subseteq \mathcal{R}^{+}$. Let $V^+(\bm{x})=0$ but $\bm{x}\notin \mathcal{R}^{+}$.
According to the concept of $\mathcal{R}^{+}$ in Definition \ref{ICC}, we have that 
\begin{equation*}
\exists \epsilon >0, \exists T\geq 0, \exists \bm{\beta}(\cdot)\in \Delta,\forall \bm{u}(\cdot)\in \mathcal{U}, \exists t\in [0,T], h(\bm{\phi}_{\bm{x}}^{\bm{u},\bm{\beta}(\bm{u})}(t))>\epsilon. 
\end{equation*}
Therefore, $\sup_{\bm{\beta}(\cdot)\in \Delta}\inf_{\bm{u}(\cdot)\in \mathcal{U}}J(\bm{x},\bm{u},\bm{\beta}(\bm{u}))\geq e^{-\gamma T}\epsilon$, which contradicts $V^{+}(\bm{x})=0$. Therefore, we conclude that $\{\bm{x}\in \mathbb{R}^n\mid V^+(\bm{x})=0\}\subseteq \mathcal{R}^{+}.$

In summary, we have that $\{\bm{x}\in \mathbb{R}^n\mid V^+(\bm{x})=0\}=\mathcal{R}^{+}.$
\end{proof}

According to Lemma \ref{inclusion}, if $V^-(\bm{x})$ and $V^+(\bm{x})$ are computed, we can obtain $\mathcal{R}^-$ and an estimation of $\mathcal{R}^+$ respectively. In order to compute them, we study  more about them and consequently exploit more properties related to them below.
\begin{lemma}
\label{cont}
Both the lower value function $V^{-}$ and the upper value function $V^{+}$ are locally Lipschitz continuous over $\mathbb{R}^n$.
\end{lemma}
\begin{proof}
We just prove the statement related to $V^-$. The one for $V^+$ can be justified following the same procedure.

Let $\epsilon>0$ and choose $\bm{\alpha}_1(\cdot)\in\Gamma$ such that
$V^-(\bm{x}_1)\geq \sup_{\bm{d}(\cdot)\in \mathcal{D}}J(\bm{x}_1,\bm{\alpha}_1(\bm{d}),\bm{d})-\epsilon.$
For $V^-(\bm{x}_2)$, we have that $V^-(\bm{x}_2)\leq \sup_{\bm{d}(\cdot)\in \mathcal{D}}J(\bm{x}_2,\bm{\alpha}_1(\bm{d}),\bm{d}).$
Moreover, we can choose $\bm{d}_1(\cdot)\in\mathcal{D}$ such that
$V^-(\bm{x}_2)\leq J(\bm{x}_2,\bm{\alpha}_1(\bm{d}_1),\bm{d}_1)+\epsilon.$
Therefore,
\begin{equation}
\begin{split}
&V^-(\bm{x}_2)-V^-(\bm{x}_1)\\
&\leq J(\bm{x}_2,\bm{\alpha}_1(\bm{d}_1),\bm{d}_1)-J(\bm{x}_1,\bm{\alpha}_1(\bm{d}_1),\bm{d}_1)+2\epsilon\\
&\leq \sup_{t\in [0,\infty)}e^{-\gamma t}h(\bm{\phi}_{\bm{x}_2}^{\bm{\alpha}_1(\bm{d}_1),\bm{d}_1}(t))-\sup_{t\in [0,\infty)}e^{-\gamma t}h(\bm{\phi}_{\bm{x}_1}^{\bm{\alpha}_1(\bm{d}_1),\bm{d}_1}(t))+2\epsilon\\
&\leq \sup_{t\in [0,\infty)}(e^{-\gamma t}h(\bm{\phi}_{\bm{x}_2}^{\bm{\alpha}_1(\bm{d}_1),\bm{d}_1}(t))-e^{-\gamma t}h(\bm{\phi}_{\bm{x}_1}^{\bm{\alpha}_1(\bm{d}_1),\bm{d}_1}(t)))+2\epsilon.
\end{split}
\end{equation}

Since $h(\bm{x})$ is bounded over $\bm{x}\in \mathbb{R}^n$,  we have that
$\lim_{t\rightarrow \infty}e^{-\gamma t}h(\bm{\phi}_{\bm{x}}^{\bm{\alpha}_1(\bm{d}_1),\bm{d}_1}(t))=0.$ As a consequence, we obtain that there exists $T>0$ such that $$e^{-\gamma t}h(\bm{\phi}_{\bm{x}_2}^{\bm{\alpha}_1(\bm{d}_1),\bm{d}_1}(t))-e^{-\gamma t}h(\bm{\phi}_{\bm{x}_1}^{\bm{\alpha}_1(\bm{d}_1),\bm{d}_1}(t))\leq \epsilon, \forall t\geq T.$$ Therefore, we infer that
\begin{equation}
\begin{split}
&V^-(\bm{x}_2)-V^-(\bm{x}_1)\\
&\leq \max\{\sup_{t\in [0,T]}(e^{-\gamma t}h(\bm{\phi}_{\bm{x}_2}^{\bm{\alpha}_1(\bm{d}_1),\bm{d}_1}(t))-e^{-\gamma t}h(\bm{\phi}_{\bm{x}_1}^{\bm{\alpha}_1(\bm{d}_1),\bm{d}_1}(t))),\\
& \sup_{t\in [T,\infty)} \{e^{-\gamma t}h(\bm{\phi}_{\bm{x}_2}^{\bm{\alpha}_1(\bm{d}_1),\bm{d}_1}(t))-e^{-\gamma t}h(\bm{\phi}_{\bm{x}_1}^{\bm{\alpha}_1(\bm{d}_1),\bm{d}_1}(t))\} \}+2\epsilon\\
&\leq L_h e^{L_{\bm{f}}T}\|\bm{x}_1-\bm{x}_2\|+3\epsilon,
\end{split}
\end{equation}
where  $L_h$ ad $L_{\bm{f}}$ are the Lipschitz constants of $h$ and $\bm{f}$ over $\Omega(B_1)=\{\bm{x}\mid \bm{x}=\bm{\phi}_{\bm{x}_0}^{\bm{\alpha}_1(\bm{d}_1),\bm{d}_1}(t),t\in [0,T],\bm{x}_0\in B_1\}$ with $B_1$ being a compact set in $\mathbb{R}^n$ covering $\bm{x}_1$ and $\bm{x}_2$ respectively. The same argument with the role of $\bm{x}_1$, $\bm{x}_2$ reversed establishes that $$V^-(\bm{x}_2)-V^-(\bm{x}_1)\geq -L_h e^{L_{\bm{f}}T}\|\bm{x}_1-\bm{x}_2\|-3\epsilon.$$ Since $\epsilon$ is arbitrary, there is a constant $L$ such that
$|V^-(\bm{x}_1)-V^-(\bm{x}_2)|\leq L\|\bm{x}_1-\bm{x}_2\|$, where $L$ is larger than or equal to $L_h e^{L_{\bm{f}}T}$.
\end{proof}

Besides the Lipschitz continuity of $V^-$ and $V^+$, both $V^-$ and $V^+$ satisfy the dynamic programming principle.
\begin{lemma}
\label{dyn}
For $\bm{x}\in \mathbb{R}^n$ and $t\geq 0$, we have
\begin{equation}
\label{dy1}
\begin{split}
V^{-}(\bm{x})=\inf_{\bm{\alpha}(\cdot)\in \Gamma}\sup_{\bm{d}(\cdot)\in \mathcal{D}}\max\{e^{-\gamma t}V^{-}(\bm{\phi}_{\bm{x}}^{\bm{\alpha(\bm{d})},\bm{d}}(t)),\sup_{\tau\in [0,t]} e^{-\gamma \tau}h(\bm{\phi}_{\bm{x}}^{\bm{\alpha(\bm{d})},\bm{d}}(\tau))\}
\end{split}
\end{equation}
and
\begin{equation}
\label{dy2}
\begin{split}
V^{+}(\bm{x})=\sup_{\bm{\beta}(\cdot)\in \Delta}\inf_{\bm{u}(\cdot)\in \mathcal{U}}\max\{e^{-\gamma t}V^{+}(\bm{\phi}_{\bm{x}}^{\bm{u},\bm{\beta}(\bm{u})}(t)),\sup_{\tau\in [0,t]} e^{-\gamma \tau} h(\bm{\phi}_{\bm{x}}^{\bm{u},\bm{\beta}(\bm{u})}(\tau))\}.
\end{split}
\end{equation}
\end{lemma}
\begin{proof}
Let
\begin{equation*} 
\begin{split}
W(\bm{x},t):=\inf_{\bm{\alpha}(\cdot)\in \Gamma}\sup_{\bm{d}(\cdot)\in \mathcal{D}}\max\{e^{-\gamma t}V^{-}(\bm{\phi}_{\bm{x}}^{\bm{\alpha}(\bm{d}),\bm{d}}(t)), \sup_{\tau\in [0,t]} e^{-\gamma \tau}h(\bm{\phi}_{\bm{x}}^{\bm{\alpha}(\bm{d}),\bm{d}}(\tau))\}.
\end{split}
\end{equation*} 

We will show that for every $\epsilon>0$, $V^-(\bm{x})\leq W(\bm{x},t)+2\epsilon$ and $V^-(\bm{x})\geq W(\bm{x},t)-3\epsilon$. Then since $\epsilon>0$ is arbitrary, $V^-(\bm{x})=W(\bm{x},t)$.

\textbf{1.} $V^-(\bm{x})\leq W(\bm{x},t)+2\epsilon$. Fix $\epsilon>0$ and choose $\bm{\alpha}_1(\cdot)\in\Gamma$ such that
\begin{equation*}
\begin{split}
W(\bm{x},t)\geq \sup_{\bm{d}_1(\cdot)\in\mathcal{D}}\max\{e^{-\gamma t}V^{-}(\bm{\phi}_{\bm{x}}^{\bm{\alpha}_1(\bm{d}_1),\bm{d}_1}(t)), \sup_{\tau\in [0,t]} e^{-\gamma \tau}h(\bm{\phi}_{\bm{x}}^{\bm{\alpha}_1(\bm{d}_1),\bm{d}_1}(\tau))\}-\epsilon.
\end{split}
\end{equation*} 
Similarly, choose $\bm{\alpha}_2(\cdot)\in \Gamma$ such that
$$V^{-}(\bm{y})\geq \sup_{\bm{d}_2(\cdot)\in \mathcal{D}}\sup_{\tau\in [t,\infty)}e^{-\gamma (\tau-t)} h(\bm{\phi}_{\bm{y}}^{\bm{\alpha}_2(\bm{d}_2),\bm{d}_2}(\tau-t))-\epsilon,$$
where $\bm{y}=\bm{\phi}_{\bm{x}}^{\bm{\alpha}_1(\bm{d}_1),\bm{d}_1}(t)$.

Let 
\begin{equation*}
\bm{d}(\tau)=
\begin{cases}
&\bm{d}_1(\tau) \text{~~~~~if } \tau \in [0,t)\\
&\bm{d}_2(\tau-t)\text{~if } \tau \in [t,\infty)
\end{cases}
\end{equation*}
and
\begin{equation}
\bm{\alpha}(\bm{d})(\tau)=
\begin{cases}
&\bm{\alpha}_1(\bm{d})(\tau) \text{~~~~~if } \tau \in [0,t)\\
&\bm{\alpha}_2(\bm{d})(\tau-t) \text{~if } \tau \in [t,\infty)
\end{cases}.
\end{equation}
It is easy to see that $\bm{\alpha}(\cdot): \mathcal{D}\mapsto \mathcal{U}$ is non-anticipative. By uniqueness, $\bm{\phi}_{\bm{x}}^{\bm{\alpha}(\bm{d}),\bm{d}}(\tau)=\bm{\phi}_{\bm{x}}^{\bm{\alpha}_1(\bm{d}_1),\bm{d}_1}(\tau)$ if $\tau\in [0,t)$, and $\bm{\phi}_{\bm{x}}^{\bm{\alpha}(\bm{d}),\bm{d}}(\tau)=\bm{\phi}_{\bm{y}}^{\bm{\alpha}_2(\bm{d}_2),\bm{d}_2}(\tau-t)$ if $\tau\in [t,\infty)$.

Hence,
\begin{equation}
\begin{split}
&W(\bm{x},t)\\
&\geq \sup_{\bm{d}_1(\cdot)\in\mathcal{D}}\sup_{\bm{d}_2(\cdot)\in \mathcal{D}}\max\{\sup_{\tau\in[t,\infty)}e^{-\gamma \tau} h(\bm{\phi}_{\bm{y}}^{\bm{\alpha}_2(\bm{d}_2),\bm{d}_2}(\tau-t)), \sup_{\tau\in [0,t]} e^{-\gamma \tau}h(\bm{\phi}_{\bm{x}}^{\bm{\alpha}_1(\bm{d}_1),\bm{d}_1}(\tau))\}-2\epsilon\\
&\geq \sup_{\bm{d}(\cdot)\in \mathcal{D}}\sup_{\tau\in [0,\infty)} e^{-\gamma \tau}h(\bm{\phi}_{\bm{x}}^{\bm{\alpha}(\bm{d}),\bm{d}}(\tau))-2\epsilon\\
&\geq V^-(\bm{x})-2\epsilon.
\end{split}
\end{equation}
Therefore, $V^-(\bm{x})\leq W(\bm{x},t)+2\epsilon$.

\textbf{2.} $V^-(\bm{x})\geq W(\bm{x},t)-3\epsilon$. Fix $\epsilon>0$ and choose $\bm{\alpha}(\cdot)\in \Gamma$ such that
\begin{equation}
\label{16}
V^-(\bm{x})\geq \sup_{\bm{d}(\cdot)\in \mathcal{D}}\sup_{t\in [0,\infty)}e^{-\gamma t}h(\bm{\phi}_{\bm{x}}^{\bm{\alpha}(\bm{d}),\bm{d}}(t))-\epsilon.
\end{equation}
By the definition of $W(\bm{x},t)$, we have
\begin{equation*}
\begin{split}
W(\bm{x},t)\leq \sup_{\bm{d}(\cdot)\in \mathcal{D}}\max\{\max_{\tau\in [0,t]}e^{-\gamma \tau}h(\bm{\phi}_{\bm{x}}^{\bm{\alpha}(\bm{d}),\bm{d}}(\tau)), e^{-\gamma t}V^-(\bm{\phi}_{\bm{x}}^{\bm{\alpha}(\bm{d}),\bm{d}}(t))\}.
\end{split}
\end{equation*}

Hence there exists $\bm{d}_1(\cdot)\in\mathcal{D}$ such that
\begin{equation}
\label{17}
W(\bm{x},t)\leq \max\{\max_{\tau\in [0,t]}e^{-\gamma \tau}h(\bm{\phi}_{\bm{x}}^{\bm{\alpha}(\bm{d}_1),\bm{d}_1}(\tau)),e^{-\gamma t}V^-(\bm{y})\}+\epsilon.
\end{equation}
where $\bm{y}=\bm{\phi}_{\bm{x}}^{\bm{\alpha}(\bm{d}_1),\bm{d}_1}(t).$
Moreover, we have
\begin{equation}
\begin{split}
V^{-}(\bm{y})\leq \sup_{\bm{d}(\cdot)\in \mathcal{D}}\sup_{\tau\in [t,\infty)} e^{-\gamma (\tau-t)} h(\bm{\phi}_{\bm{y}}^{\bm{\alpha}(\bm{d}),\bm{d}}(\tau-t)), \forall \tau\in [t,\infty).
\end{split}
\end{equation}
so there exists $\bm{d}_2(\cdot)\in \mathcal{D}$ such that
\begin{equation}
\label{18}
\begin{split}
V^{-}(\bm{y})\leq \sup_{\tau\in [t,\infty)} e^{-\gamma (\tau-t)} h(\bm{\phi}_{\bm{y}}^{\bm{\alpha}(\bm{d}_2),\bm{d}_2}(\tau-t))+\epsilon.
\end{split}
\end{equation}
We define
\begin{equation}
\bm{d}(\tau)=
\begin{cases}
&\bm{d}_1(\tau) \text{ if }\tau \in [0,t)\\
&\bm{d}_2(\tau-t) \text{ if }\tau \in [t,\infty)
\end{cases}.
\end{equation}
Therefore,  combining \eqref{17} and \eqref{18}, we have
\[W(\bm{x},t)\leq \sup_{\tau\in [0,\infty)}e^{-\gamma \tau}h(\bm{\phi}_{\bm{x}}^{\bm{\alpha}(\bm{d}),\bm{d}}(\tau))+2\epsilon,\]
which together with \eqref{16} implies $V^-(\bm{x})\geq W(\bm{x},t)-3\epsilon$.

The above procedure can be applied to prove that $V^{+}$ satisfies the dynamic programming principle \eqref{dy2}.
\end{proof}

Based on the established dynamic programming principle in Lemma \ref{dyn}, we construct Hamilton-Jacobi partial differential equations associated with $V^-$ and $V^+$ respectively,
 \begin{equation}
 \label{HJB1}
 \min\{\gamma V(\bm{x})-H^-(\bm{x},\frac{\partial V(\bm{x})}{\partial \bm{x}}),V(\bm{x})-h(\bm{x})\}=0~\text{and}
 \end{equation}
 \begin{equation}
 \label{HJB2}
 \min\{\gamma V(\bm{x})-H^+(\bm{x},\frac{\partial V(\bm{x})}{\partial \bm{x}}),V(\bm{x})-h(\bm{x})\}=0,
 \end{equation}
where
\begin{align}
&H^-(\bm{x},\bm{p})=\sup_{\bm{d}\in D}\inf_{\bm{u}\in U} \bm{p}\cdot \bm{f}(\bm{x},\bm{u},\bm{d})\text{~and}\label{upper_h}\\
&H^{+}(\bm{x},\bm{p})=\inf_{\bm{u}\in U}\sup_{\bm{d}\in D} \bm{p}\cdot \bm{f}(\bm{x},\bm{u},\bm{d}) \label{lower_h}
\end{align}
are the sup-inf and inf-sup Hamiltonians respectively.
These two equations are the core focus of this paper. We in the sequel will show that $V^-$ and $V^+$ are respectively the unique Lipschitz continuous and bounded viscosity solution to \eqref{HJB1} and \eqref{HJB2}. Before this, we first recall the concept of viscosity solutions to \eqref{HJB1} (or \eqref{HJB2}).

\begin{definition} \cite{Bardi1997}
\label{viscosity}
A locally bounded continuous function $V(\bm{x})$ on $\mathbb{R}^n$ is a viscosity solution of \eqref{HJB1} (\eqref{HJB2}), if 1) for any test function $v\in C^{\infty}(\mathbb{R}^n)$ such that $V-v$ attains a local minimum at $\bm{x}_0\in \mathbb{R}^n$,
\begin{equation}
\label{super}
\begin{split}
\min\big\{\gamma V(\bm{x}_0)-H^-(\bm{x}_0,\frac{\partial v(\bm{x})}{\partial \bm{x}}\mid_{\bm{x}=\bm{x}_0}),
V(\bm{x}_0)-h(\bm{x}_0)\big\}\geq 0\\
(\min\big\{\gamma V(\bm{x}_0)-H^+(\bm{x}_0,\frac{\partial v(\bm{x})}{\partial \bm{x}}\mid_{\bm{x}=\bm{x}_0}),
V(\bm{x}_0)-h(\bm{x}_0)\big\}\geq 0)
\end{split}
\end{equation}
holds (i.e., $V$ is a viscosity supersolution); 2) for any test function $v\in C^{\infty}(\mathbb{R}^n)$ such that $V-v$ attains a local maximum at $\bm{x}_0 \in \mathbb{R}^n$,
\begin{equation}
\label{sub}
\begin{split}
\min\big\{\gamma V(\bm{x}_0)-H^-(\bm{x}_0,\frac{\partial v(\bm{x})}{\partial \bm{x}}\mid_{\bm{x}=\bm{x}_0}), V(\bm{x}_0)-h(\bm{x}_0)\big\}\leq 0\\
(\min\big\{\gamma V(\bm{x}_0)-H^+(\bm{x}_0,\frac{\partial v(\bm{x})}{\partial \bm{x}}\mid_{\bm{x}=\bm{x}_0}), V(\bm{x}_0)-h(\bm{x}_0)\big\}\leq 0)
\end{split}
\end{equation}
holds (i.e., $V$ is a viscosity subsolution).
\end{definition}

In order to prove that $V^-(\bm{x})$ and $V^+(\bm{x})$ are respectively the viscosity solution to \eqref{HJB1} and \eqref{HJB2}, we need an intermediate lemma below.
\begin{lemma}
\label{smooth}
Let $v\in C^{\infty}(\mathbb{R}^n)$.
\begin{enumerate}
\item If $\gamma v(\bm{x}_0)-H^{-}(\bm{x}_0,\frac{\partial v(\bm{x})}{\partial \bm{x}}\mid_{\bm{x}=\bm{x}_0})\leq -\theta<0$, then, for sufficiently small $\delta>0$, there exists $\bm{d}(\cdot)\in \mathcal{D}$ such that for all $\bm{\alpha}(\cdot)\in \Gamma$ and all $s \in [0,\delta]$,
\begin{equation*}
\begin{split}
\gamma v(\bm{\phi}_{\bm{x}_0}^{\bm{\alpha}(\bm{d}),\bm{d}}(s))-\frac{\partial v(\bm{x})}{\partial \bm{x}}\mid_{\bm{x}=\bm{\phi}_{\bm{x}_0}^{\bm{\alpha}(\bm{d}),\bm{d}}(s)}\cdot  \bm{f}(\bm{\phi}_{\bm{x}_0}^{\bm{\alpha}(\bm{d}),\bm{d}}(s),\bm{\alpha}(\bm{d})(s),\bm{d}(s))\leq -\frac{\theta}{2}.
\end{split}
\end{equation*}

\item If $\gamma v(\bm{x}_0)-H^{-}(\bm{x}_0,\frac{\partial v(\bm{x})}{\partial \bm{x}}\mid_{\bm{x}=\bm{x}_0})\geq \theta>0$, then, for sufficiently small $\delta>0$, there exists $\bm{\alpha}(\cdot)\in \Gamma$ such that for all $\bm{d}(\cdot)\in \mathcal{D}$ and all $s \in [0,\delta]$,
\begin{equation*}
\begin{split}
\gamma v(\bm{\phi}_{\bm{x}_0}^{\bm{\alpha}(\bm{d}),\bm{d}}(s))-\frac{\partial v(\bm{x})}{\partial \bm{x}}\mid_{\bm{x}=\bm{\phi}_{\bm{x}_0}^{\bm{\alpha}(\bm{d}),\bm{d}}(s)}\cdot \bm{f}(\bm{\phi}_{\bm{x}_0}^{\bm{\alpha}(\bm{d}),\bm{d}}(s),\bm{\alpha}(\bm{d})(s),\bm{d}(s))\geq \frac{\theta}{2}.
\end{split}
\end{equation*}

\item If $\gamma v(\bm{x}_0)-H^{+}(\bm{x}_0,\frac{\partial v(\bm{x})}{\partial \bm{x}}\mid_{\bm{x}=\bm{x}_0})\geq \theta>0$, then, for sufficiently small $\delta>0$, there exists $\bm{u}(\cdot)\in \mathcal{U}$
such that for all $\bm{\beta}(\cdot)\in \Delta$ and all $s \in [0,\delta]$,
\begin{equation*}
\begin{split}
\gamma v(\bm{\phi}_{\bm{x}_0}^{\bm{u},\bm{\beta}(\bm{u})}(s))- \frac{\partial v(\bm{x})}{\partial \bm{x}}\mid_{\bm{x}=\bm{\phi}_{\bm{x}_0}^{\bm{u},\bm{\beta}(\bm{u})}}\cdot \bm{f}(\bm{\phi}_{\bm{x}_0}^{\bm{u},\bm{\beta}(\bm{u})}(s),\bm{u}(s),\bm{\beta}(\bm{u})(s))\geq \frac{\theta}{2}.
\end{split}
\end{equation*}
\item If $\gamma v(\bm{x}_0)-H^{+}(\bm{x}_0,\frac{\partial v(\bm{x})}{\partial \bm{x}}\mid_{\bm{x}=\bm{x}_0})\leq- \theta<0$, then, for sufficiently small $\delta>0$, there exists $\bm{\beta}(\cdot)\in \Delta$ such that for all $\bm{u}(\cdot) \in \mathcal{U}$ and all $s \in [0,\delta]$,
\begin{equation*}
\begin{split}
\gamma v(\bm{\phi}_{\bm{x}_0}^{\bm{u},\bm{\beta}(\bm{u})}(s))-\frac{\partial v(\bm{x})}{\partial \bm{x}}\mid_{\bm{x}=\bm{\phi}_{\bm{x}_0}^{\bm{u},\bm{\beta}(\bm{u})}}\cdot \bm{f}(\bm{\phi}_{\bm{x}_0}^{\bm{u},\bm{\beta}(\bm{u})}(s),\bm{u}(s),\bm{\beta}(\bm{u})(s))\leq -\frac{\theta}{2}.
\end{split}
\end{equation*}
\end{enumerate}
\end{lemma}
\begin{proof}
The proofs of statements  1 and 2 are given. The statements  3 and 4 can be justified similarly.

1. Since $\gamma v(\bm{x}_0)-H^{-}(\bm{x}_0,\frac{\partial v(\bm{x})}{\partial \bm{x}}\mid_{\bm{x}=\bm{x}_0})\leq -\theta<0$, there exists $\bm{d}_0\in D$ such that $$\gamma v(\bm{x}_0)- \frac{\partial v(\bm{x})}{\partial \bm{x}}\mid_{\bm{x}=\bm{x}_0}\cdot \bm{f}(\bm{x}_0,\bm{u},\bm{d}_0)\leq -\frac{3}{4}\theta<0, \forall \bm{u}\in U.$$ Also, since $v\in C^{\infty}$, $\bm{f}(\bm{x},\bm{u},\bm{d})$ is continuous over $(\bm{x}, \bm{u},\bm{d})$, there exists $\delta_{\bm{u}}$ for $\bm{u}\in U$ such that for $\bm{x} \text{ satisfying }\|\bm{x}-\bm{x}_0\|\leq \delta_{\bm{u}}$,
\[
 \gamma v(\bm{x})-\frac{\partial v(\bm{x})}{\partial \bm{x}}\cdot \bm{f}(\bm{x},\bm{u},\bm{d}_0)\leq -\frac{3}{5}\theta<0.
\] 
Since $U$ is a compact set in $\mathbb{R}^m$, there exist finitely many distinct points $\bm{u}_{1},\ldots, \bm{u}_{l}\in U$ with positive values $\delta_{\bm{u}_1},\ldots,\delta_{\bm{u}_l}$ such that $$U\subset \cup_{i=1}^l\{\bm{u}\mid \|\bm{u}-\bm{u}_{i}\|\leq \delta_{\bm{u}_i}\}$$ and 
$$\gamma v(\bm{x})-\frac{\partial v(\bm{x})}{\partial \bm{x}}\cdot \bm{f}(\bm{x},\bm{u},\bm{d}_0)\leq -\frac{1}{2}\theta<0$$ for $\bm{x}$ satisfying $\|\bm{x}-\bm{x}_0\|\leq \delta_i$ and $\bm{u}$ satisfying $\|\bm{u}-\bm{u}_{i}\|\leq \delta_{\bm{u}_i}$, where $i=1,\ldots,l$. Therefore, 
$$\gamma v(\bm{x})-\frac{\partial v(\bm{x})}{\partial \bm{x}}\cdot \bm{f}(\bm{x},\bm{u},\bm{d}_0)\leq -\frac{1}{2}\theta<0, \forall \bm{u}\in U$$ for $\bm{x} \text{ satisfying }\|\bm{x}-\bm{x}_0\|\leq \delta'=\min_{i=1,\ldots,l}\delta_{\bm{u}_i}$.

Let $\Omega$ be a compact set in $\mathbb{R}^n$ which covers all states traversed by trajectories starting from $\bm{x}_0$ within a finite time interval $[0,\delta^{''}]$, and $M$ be the upper bound of $\bm{f}(\bm{x},\bm{u},\bm{d})$ over $\Omega\times U\times D$. We have 
\begin{equation*}
\|\bm{\phi}_{\bm{x}_0}^{\bm{u},\bm{d}}(t)-\bm{x}_0\|=\int_{\tau=0}^{t}\|\bm{f}(\bm{x}(\tau),\bm{u}(\tau),\bm{d}(\tau))\|d\tau\leq M t, \forall \bm{u}(\cdot)\in \mathcal{U}, \forall \bm{d}(\cdot)\in \mathcal{D}.
\end{equation*}
Therefore, there exists $\delta>0$ such that 
\begin{equation}
\label{delta}
\|\bm{\phi}_{\bm{x}_0}^{\bm{u},\bm{d}}(t)-\bm{x}_0\|\leq \delta', \forall t\in [0,\delta], \forall \bm{u}(\cdot)\in \mathcal{U}, \forall \bm{d}(\cdot)\in \mathcal{D}.
\end{equation}

We choose a measurable function $\bm{d}':[0,\infty)\mapsto D$ with $\bm{d}'(s)=\bm{d}_0$ for $s\in [0,\infty)$. Obviously, $\bm{d}'(\cdot)\in \mathcal{D}$. Therefore, we have
\[\gamma v(\bm{\phi}_{\bm{x}_0}^{\bm{u},\bm{d}'}(s))- \frac{\partial v(\bm{x})}{\partial \bm{x}}\mid_{\bm{x}=\bm{\phi}_{\bm{x}_0}^{\bm{u},\bm{d}'}(s)}\cdot\bm{f}(\bm{\phi}_{\bm{x}_0}^{\bm{u},\bm{d}'}(s),\bm{u}(s),\bm{d}'(s))\leq -\frac{\theta}{2}, \forall \bm{u}(\cdot)\in \mathcal{U},\forall s \in [0,\delta],\]
implying that for all $\bm{\alpha}(\cdot)\in \Gamma$ and all $s \in [0,\delta]$, 
\begin{equation*}
\begin{split}
\gamma v(\bm{\phi}_{\bm{x}_0}^{\bm{\alpha}(\bm{d}'),\bm{d}'}(s))- \frac{\partial v(\bm{x})}{\partial \bm{x}}\mid_{\bm{x}=\bm{\phi}_{\bm{x}_0}^{\bm{\alpha}(\bm{d}'),\bm{d}'}(s)}\cdot \bm{f}(\bm{\phi}_{\bm{x}_0}^{\bm{\alpha(\bm{d}')},\bm{d}'}(s),\bm{\alpha}(\bm{d}')(s),\bm{d}'(s))\leq -\frac{\theta}{2}.
\end{split}
\end{equation*}

2. Since $\gamma v(\bm{x}_0)-H^{-}(\bm{x}_0,\frac{\partial v(\bm{x})}{\partial \bm{x}}\mid_{\bm{x}=\bm{x}_0})\geq \theta>0$, there exists a corresponding $\bm{u}_{\bm{d}_0} \in U$ for every $\bm{d}_0\in D$ such that
$$\gamma v(\bm{x}_0)-\frac{\partial v(\bm{x})}{\partial \bm{x}}\mid_{\bm{x}=\bm{x}_0}\cdot \bm{f}(\bm{x}_0,\bm{u}_{\bm{d}_0},\bm{d}_0)\geq \frac{3}{4}\theta>0.$$ Since $v\in C^{\infty}$ and $\bm{f}(\bm{x},\bm{u},\bm{d})$ is continuous over $\bm{x}$, $\bm{u}$ and $\bm{d}$, there exists $\delta'>0$ such that
for $\bm{d}\in D$ satisfying $\|\bm{d}-\bm{d}_0\|\leq \delta'$ and $\bm{x}$ satisfying $\|\bm{x}-\bm{x}_0\|\leq \delta'$,
$$\gamma v(\bm{x})-\frac{\partial v(\bm{x})}{\partial \bm{x}}\cdot \bm{f}(\bm{x},\bm{u}_{\bm{d}_0},\bm{d})\geq \frac{3}{5}\theta>0.$$
Since $D$ is a compact set in $\mathbb{R}^l$, there exist finitely many distinct points $\bm{d}_1,\ldots, \bm{d}_l\in D$ with positive values $\delta_1,\ldots,\delta_l$ such that $D\subset \cup_{i=1}^l\{\bm{d}\mid \|\bm{d}-\bm{d}_i\|\leq \delta_i\}.$ Moreover, there exists $\bm{u}_{\bm{d}_i}\in U$ such that for $\bm{d}$ satisfying $\|\bm{d}-\bm{d}_i\|\leq \delta_i$ and $\bm{x}$ satisfying $\|\bm{x}-\bm{x}_0\|\leq \delta_i$,
$$\gamma v(\bm{x})-\frac{\partial v(\bm{x})}{\partial \bm{x}}\cdot \bm{f}(\bm{x},\bm{u}_{\bm{d}_i},\bm{d})\geq \frac{1}{2}\theta>0$$ holds, where $i=1,\ldots,l$.

Setting $\bm{\nu}: D\mapsto U$ such that $\bm{\nu}(\bm{d})=\bm{u}_{\bm{d}_i}$ if $\bm{d}\in \{\bm{d}\mid \|\bm{d}-\bm{d}_i\|\leq \delta_i\}\setminus \cup_{j=1}^{i-1} \{\bm{d}\mid \|\bm{d}-\bm{d}_j\|\leq \delta_j\}$ for $i=1,\ldots,l$, we have that for $\bm{x}$ satisfying $\|\bm{x}-\bm{x}_0\|\leq \delta'=\min_{i=1,\ldots,l} \delta_i$,
$$\gamma v(\bm{x})-\frac{\partial v(\bm{x})}{\partial \bm{x}}\cdot \bm{f}(\bm{x},\bm{\nu}(\bm{d}),\bm{d})\geq \frac{1}{2}\theta>0, \forall \bm{d}\in D.$$
Furthermore, like \eqref{delta}, we obtain that there exists $\delta>0$ such that
$$\bm{\phi}_{\bm{x}_0}^{\bm{\nu}(\bm{d}),\bm{d}}(s) \in \{\bm{x}\mid \|\bm{x}-\bm{x}_0\|\leq \delta'\}, \forall s\in [0,\delta], \forall \bm{d}(\cdot)\in \mathcal{D}.$$
Let $\bm{\alpha}(\cdot): \mathcal{D}\mapsto \mathcal{U}$ be $\bm{\alpha}(\bm{d})(s)=\bm{\nu}(\bm{d}(s))$ for $s\geq 0$. It is obvious that $\bm{\alpha}(\cdot)\in \Gamma$. Consequently, there exist $\delta>0$ and a strategy $\bm{\alpha}(\cdot)\in \Gamma$ such that for all $\bm{d}(\cdot)\in \mathcal{D}$ and all $s \in [0,\delta]$,
\begin{equation*}
\begin{split}
\gamma v(\bm{\phi}_{\bm{x}_0}^{\bm{\alpha}(\bm{d}),\bm{d}}(s))- \frac{\partial v(\bm{x})}{\partial \bm{x}}\mid_{\bm{x}=\bm{\phi}_{\bm{x}_0}^{\bm{\alpha}(\bm{d}),\bm{d}}(s)}\cdot \bm{f}(\bm{\phi}_{\bm{x}_0}^{\bm{\alpha}(\bm{d}),\bm{d}}(s),\bm{\alpha}(\bm{d})(s),\bm{d}(s))\geq \frac{\theta}{2}.
\end{split}
\end{equation*}
\end{proof}

We in the following reduce $V^-(\bm{x})$ and $V^+(\bm{x})$ to the viscosity solution to \eqref{HJB1} and \eqref{HJB2} respectively. 
 \begin{theorem}
 \label{viscosity1}
 $V^{-}$ and $V^{+}$ are respectively the viscosity solution to Hamilton-Jacobi equations  \eqref{HJB1} and \eqref{HJB2}.
 \end{theorem}
 \begin{proof}
Likewise, we just prove the statement pertinent to $V^-$.
We will prove that $V^{-}$ is both viscosity sub and super-solution to \eqref{HJB1} according to Definition \ref{viscosity}.

Firstly, we prove that $V^{-}$ is a sub-solution to \eqref{HJB1}. Let $v\in C^{\infty}(\mathbb{R}^n)$ such that $V^{-}-v$ attains a local maximum at $\bm{x}_0$. Without loss of generality, assume that this maximum is zero, i.e. $V^{-}(\bm{x}_0)=v(\bm{x}_0)$. According to the continuity of $V^-(\bm{x})$ and $v(\bm{x})$, there exists a positive value $\overline{\delta}$ such that $$V^{-}(\bm{x})-v(\bm{x})\leq 0$$ for $\bm{x}$ satisfying $\|\bm{x}-\bm{x}_0\|\leq \overline{\delta}$. Suppose \eqref{sub} is false. Then there definitely exists a positive value $\epsilon_1$ such that
\begin{equation}
\label{1}
h(\bm{x}_0)\leq v(\bm{x}_0)-\epsilon_1~\text{and}
\end{equation}
 \begin{equation}
 \label{2}
 \gamma v(\bm{x}_0)-H^{-}(\bm{x}_0,\frac{\partial v(\bm{x})}{\partial \bm{x}}\mid_{\bm{x}=\bm{x}_0})\geq \epsilon_1
 \end{equation}
hold. Therefore, for the former inequality, i.e., $h(\bm{x}_0)\leq v(\bm{x}_0)-\epsilon_1$, there exists a sufficiently small $\delta_1>0$ with $\delta_1\leq \overline{\delta}$ such that for $\bm{x}$ satisfying $\|\bm{x}-\bm{x}_0\|\leq \delta_1$ and $t$ satisfying $0 \leq t \leq \delta_1$, $$e^{-\gamma t}h(\bm{x})\leq v(\bm{x}_0)-\frac{\epsilon_1}{2}.$$ According to Lemma \ref{smooth}, \eqref{2} implies that for sufficiently small $\delta>0$, there exists
a strategy $\bm{\alpha}_1(\cdot)\in \Gamma$ such that for all $\bm{d}(\cdot)\in \mathcal{D}$ and all $s \in [0,\delta]$,
\begin{equation}
\label{3}
\begin{split}
\gamma v(\bm{\phi}_{\bm{x}_0}^{\bm{\alpha}_1(\bm{d}),\bm{d}}(s))-\frac{\partial v(\bm{x})}{\partial \bm{x}}\mid_{\bm{x}=\bm{\phi}_{\bm{x}_0}^{\bm{\alpha}_1(\bm{d}),\bm{d}}(s)}\cdot \bm{f}(\bm{\phi}_{\bm{x}_0}^{\bm{\alpha}_1(\bm{d}),\bm{d}}(s),\bm{\alpha}_1(\bm{d})(s),\bm{d}(s))\geq \frac{\epsilon_1}{2}.
\end{split}
\end{equation}
$\delta$ can be chosen such that $\|\bm{\phi}_{\bm{x}_0}^{\bm{\alpha}_1(\bm{d}),\bm{d}}(s)-\bm{x}_0\|\leq \delta_1, \forall s \in [0,\delta], \forall \bm{d}(\cdot)\in \mathcal{D}.$ Since $v\in C^{\infty}(\mathbb{R}^n)$, by applying Gr\"{o}nwall's inequality \cite{gronwall1919} to \eqref{3} with the time interval [0, $\delta$], we have
\begin{equation}
v(\bm{\phi}_{\bm{x}_0}^{\bm{\alpha}_1(\bm{d}),\bm{d}}(\delta))\leq e^{\delta \gamma}v(\bm{x}_0)+\frac{\epsilon_1}{2\gamma}(1-e^{\delta \gamma}).
\end{equation}
Therefore,
\begin{equation}
e^{-\delta \gamma}v(\bm{\phi}_{\bm{x}_0}^{\bm{\alpha}_1(\bm{d}),\bm{d}}(\delta))\leq v(\bm{x}_0)-\frac{\epsilon_1}{2\gamma}(1-e^{-\delta \gamma}).
\end{equation}
Furthermore, since $V^{-}(\bm{x})\leq v(\bm{x})$ for $\bm{x}$ satisfying $\|\bm{x}-\bm{x}_0\|\leq \delta_1$ with $V^-(\bm{x}_0)=v(\bm{x}_0)$ as well as $V^{-}\geq 0$, we have 
\[e^{-\delta \gamma}V^-(\bm{\phi}_{\bm{x}_0}^{\bm{\alpha}_1(\bm{d}),\bm{d}}(\delta))\leq V^-(\bm{x}_0)-\frac{\epsilon_1}{2\gamma}(1-e^{-\delta \gamma}).\] Therefore, according to \eqref{dy1}, we finally have
\begin{equation}
\label{4}
\begin{split}
V^-(\bm{x}_0)&=\inf_{\bm{\alpha}(\cdot)\in \Gamma}\sup_{\bm{d}(\cdot)\in \mathcal{D}}\max\{e^{-\gamma \delta}V^{-}(\bm{\phi}_{\bm{x}_0}^{\bm{\alpha}(\bm{d}),\bm{d}}(\delta)),\sup_{\tau\in [0,\delta]} e^{-\gamma \tau}h(\bm{\phi}_{\bm{x}_0}^{\bm{\alpha}(\bm{d}),\bm{d}}(\tau))\}\\
&\leq \sup_{\bm{d}(\cdot)\in \mathcal{D}}\max\{e^{-\gamma \delta}V^{-}(\bm{\phi}_{\bm{x}_0}^{\bm{\alpha}_1(\bm{d}),\bm{d}}(\delta)),\sup_{\tau\in [0,\delta]} e^{-\gamma \tau}h(\bm{\phi}_{\bm{x}_0}^{\bm{\alpha}_1(\bm{d}),\bm{d}}(\tau))\}\\
&\leq \max\{e^{-\gamma \delta}V^{-}(\bm{\phi}_{\bm{x}_0}^{\bm{\alpha}_1(\bm{d}_1),\bm{d}_1}(\delta)),\sup_{\tau\in [0,\delta]} e^{-\gamma \tau}h(\bm{\phi}_{\bm{x}_0}^{\bm{\alpha}_1(\bm{d}_1),\bm{d}_1}(\tau))\}+\epsilon_3\\
&\leq V^-(\bm{x}_0)-\min\{\frac{\epsilon_1}{2},\frac{\epsilon_1}{2\gamma}(1-e^{-\delta \gamma})\}+\epsilon_3\\
&<V^-(\bm{x}_0),
\end{split}
\end{equation}
which is a contradiction. In \eqref{4}, $\bm{d}_1(\cdot)\in \mathcal{D}$ satisfies
\begin{equation}
\begin{split}
&\sup_{\bm{d}(\cdot)\in \mathcal{D}}\max\{e^{-\gamma \delta}V^{-}(\bm{\phi}_{\bm{x}_0}^{\bm{\alpha}_1(\bm{d}),\bm{d}}(\delta)),\sup_{\tau\in [0,\delta]} e^{-\gamma \tau}h(\phi_{\bm{x}_0}^{\bm{\alpha}_1(\bm{d}),\bm{d}}(\tau))\}\\
&\leq \max\{e^{-\gamma \delta}V^{-}(\bm{\phi}_{\bm{x}_0}^{\bm{\alpha}_1(\bm{d}_1),\bm{d}_1}(\delta)),\sup_{\tau\in [0,\delta]} e^{-\gamma \tau}h(\phi_{\bm{x}_0}^{\bm{\alpha}_1(\bm{d}_1),\bm{d}_1}(\tau))\}+\epsilon_3
\end{split}
\end{equation}
with $0<\epsilon_3<\min\{\frac{\epsilon_1}{2},\frac{\epsilon_1}{2\gamma}(1-e^{-\delta \gamma})\}$. Consequently, $V^{-}$ is a subsolution to \eqref{HJB1}.

In what follows we prove that $V^{-}$ is a super-solution to \eqref{HJB1}. Let $v\in C^{\infty}(\mathbb{R}^n)$ such that $V^{-}-v$ attains a local minimum at $\bm{x}_0$. Without loss of generality, assume that this minimum is zero, i.e., $V^{-}(\bm{x}_0)=v(\bm{x}_0)$. Therefore, there exists a positive value $\overline{\delta}$ such that $V^-(\bm{x})-v(\bm{x})\geq 0$ for $\bm{x}$ satisfying $\|\bm{x}-\bm{x}_0\|\leq \overline{\delta}$.
Assume that  \eqref{super} is false. Since $V^{-}(\bm{x})\geq h(\bm{x})$ for $\bm{x}\in \mathbb{R}^n$ according to \eqref{dy1}, $v(\bm{x}_0)\geq h(\bm{x}_0)$ holds. Therefore,
\begin{equation}
\label{5}
\gamma v(\bm{x}_0)-H^-(\bm{x}_0,\frac{\partial v(\bm{x})}{\partial \bm{x}}\mid_{\bm{x}=\bm{x}_0})<0
\end{equation}
holds, i.e., there exists a positive value $\theta>0$ such that
\begin{equation}
\label{6}
\gamma v(\bm{x}_0)-H^-(\bm{x}_0,\frac{\partial v(\bm{x})}{\partial \bm{x}}\mid_{\bm{x}=\bm{x}_0})<-\theta.
\end{equation}
According to Lemma \ref{smooth}, we have that for sufficiently small $\delta>0$, there exists $\bm{d}_1(\cdot)\in \mathcal{D}$ such that for all strategies $\bm{\alpha}(\cdot)\in \Gamma$ and all $s \in [0,\delta]$,
\begin{equation}
\label{7}
\begin{split}
\gamma v(\bm{\phi}_{\bm{x}_0}^{\bm{\alpha}(\bm{d}_1),\bm{d}_1}(s))-\frac{\partial v(\bm{x})}{\partial \bm{x}}\mid_{\bm{x}=\bm{\phi}_{\bm{x}_0}^{\bm{\alpha}(\bm{d}_1),\bm{d}_1}(s)}\cdot \bm{f}(\bm{\phi}_{\bm{x}_0}^{\bm{\alpha}(\bm{d}_1),\bm{d}_1}(s),\bm{\alpha}(\bm{d}_1)(s),\bm{d}_1(s))\leq -\frac{\theta}{2}.
\end{split}
\end{equation}
$\delta$ can be chosen such that $\|\bm{\phi}_{\bm{x}_0}^{\bm{\alpha}(\bm{d}_1),\bm{d}_1}(s)-\bm{x}_0\|\leq \overline{\delta}, \forall s \in [0,\delta], \forall \bm{\alpha}(\cdot)\in \Gamma.$

By applying Gr\"{o}nwall's inequality \cite{gronwall1919} to \eqref{7} with the time interval [0, $\delta$], we obtain 
\begin{equation}
v(\bm{\phi}_{\bm{x}_0}^{\bm{\alpha}(\bm{d}_1),\bm{d}_1}(\delta))\geq e^{\delta \gamma}v(\bm{x}_0)-\frac{\theta}{2\gamma}(1-e^{\delta \gamma}).
\end{equation}
Therefore,
\begin{equation}
e^{-\delta \gamma}v(\bm{\phi}_{\bm{x}_0}^{\bm{\alpha}(\bm{d}_1),\bm{d}_1}(\delta))\geq v(\bm{x}_0)+\frac{\theta}{2\gamma}(1-e^{-\delta \gamma}).
\end{equation}
Furthermore, since $V^{-}\geq v$ for $\bm{x}\in \{\bm{x}\mid \|\bm{x}-\bm{x}_0\|\leq \overline{\delta}\}$ with $V^-(\bm{x}_0)=v(\bm{x}_0)$ as well as $V^-(\bm{x})\geq 0$ over $\bm{x}\in \mathbb{R}^n$, we have 
\[e^{-\delta \gamma}V^-(\bm{\phi}_{\bm{x}_0}^{\bm{\alpha}(\bm{d}_1),\bm{d}_1}(\delta))\geq V^-(\bm{x}_0)+\frac{\theta}{2\gamma}(1-e^{-\delta \gamma}).\] Therefore, according to \eqref{dy1}, we finally have
\begin{equation}
\label{8}
\begin{split}
V^-(\bm{x}_0)&=\inf_{\bm{\alpha}(\cdot)\in \Gamma}\sup_{\bm{d}(\cdot)\in \mathcal{D}}\max\{e^{-\gamma \delta}V^{-}(\bm{\phi}_{\bm{x}_0}^{\bm{\alpha}(\bm{d}),\bm{d}}(\delta)),\sup_{\tau\in [0,\delta]} e^{-\gamma \tau}h(\bm{\phi}_{\bm{x}_0}^{\bm{\alpha}(\bm{d}),\bm{d}}(\tau))\}\\
&\geq \sup_{\bm{d}(\cdot)\in \mathcal{D}}\max\{e^{-\gamma \delta}V^{-}(\bm{\phi}_{\bm{x}_0}^{\bm{\alpha}_1(\bm{d}),\bm{d}}(\delta)),\sup_{\tau\in [0,\delta]} e^{-\gamma \tau}h(\bm{\phi}_{\bm{x}_0}^{\bm{\alpha}_1(\bm{d}),\bm{d}}(\tau))\}-\epsilon_1\\
&\geq \max\{e^{-\gamma \delta}V^{-}(\bm{\phi}_{\bm{x}_0}^{\bm{\alpha}_1(\bm{d}_1),\bm{d}_1}(\delta)),\sup_{\tau\in [0,\delta]} e^{-\gamma \tau}h(\bm{\phi}_{\bm{x}_0}^{\bm{\alpha}_1(\bm{d}_1),\bm{d}_1}(\tau))\}-\epsilon_1\\
&\geq V^-(\bm{x}_0)+\frac{\theta}{2\gamma}(1-e^{-\delta \gamma})-\epsilon_1>V^-(\bm{x}_0),
\end{split}
\end{equation}
which is a contradiction. In \eqref{8}, $\bm{\alpha}_1(\cdot)\in\Gamma$ satisfies
\begin{equation}
\begin{split}
&\inf_{\bm{\alpha}(\cdot)\in \Gamma}\sup_{\bm{d}(\cdot)\in \mathcal{D}}\max\{e^{-\gamma \delta}V^{-}(\bm{\phi}_{\bm{x}_0}^{\bm{\alpha}(\bm{d}),\bm{d}}(\delta)),\sup_{\tau\in [0,\delta]} e^{-\gamma \tau}h(\bm{\phi}_{\bm{x}_0}^{\bm{\alpha}(\bm{d}),\bm{d}}(\tau))\}\\
&\geq \sup_{\bm{d}(\cdot)\in \mathcal{D}}\max\{e^{-\gamma \delta}V^{-}(\bm{\phi}_{\bm{x}_0}^{\bm{\alpha}_1(\bm{d}),\bm{d}}(\delta)),\sup_{\tau\in [0,\delta]} e^{-\gamma \tau}h(\bm{\phi}_{\bm{x}_0}^{\bm{\alpha}_1(\bm{d}),\bm{d}}(\tau))\}-\epsilon_1
\end{split}
\end{equation}
with $0<\epsilon_1<\frac{\theta}{2\gamma}(1-e^{-\delta \gamma})$. Thus, $V^-$ is a supersolution to \eqref{HJB1}.

Therefore, we conclude that $V^-$ is a viscosity solution to \eqref{HJB1}.
 \end{proof}
   
Furthermore, we show the uniqueness of the Lipschitz continuous and bounded viscosity solutions to \eqref{HJB1} and \eqref{HJB2}. 
 \begin{theorem}
 \label{unique1}
 $V^{-}$ and $V^{+}$ are respectively the unique bounded and Lipschitz continuous viscosity solution to \eqref{HJB1} and \eqref{HJB2}.
 \end{theorem}
 
\begin{proof}
We just show the uniqueness of the Lipschitz continuous and bounded viscosity solution to \eqref{HJB1}. We first prove a comparison principle: If $V_1$ and $V_2$ are bounded Lipschitz continuous functions over $\bm{x}\in \mathbb{R}^n$, and they are respectively a viscosity sub and supersolution to \eqref{HJB1}, then $V_1\leq V_2$ in $\mathbb{R}^n$. Obviously, if such comparison principle holds, the uniqueness of bounded Lipschitz continuous solutions to \eqref{HJB1} is guaranteed. For ease of exposition, we define  $H^-(\overline{\bm{x}})=H^-(\overline{\bm{x}},\frac{\partial \phi(\bm{x})}{\partial \bm{x}}\mid_{\bm{x}=\overline{\bm{x}}})$ and $H^-(\overline{\bm{y}})=H^-(\overline{\bm{y}},\frac{\partial \psi(\bm{y})}{\partial \bm{y}}\mid_{\bm{y}=\overline{\bm{y}}}).$

Let
\[\Phi(\bm{x},\bm{y})=V_1(\bm{x})-V_2(\bm{y})-\frac{\|\bm{x}-\bm{y}\|^2}{2\epsilon}-\delta (\langle\bm{x}\rangle^m+\langle\bm{y}\rangle^m),\]
where $\langle\bm{x}\rangle=(1+\|\bm{x}\|^2)^{\frac{1}{2}}$, and $\epsilon,\delta,m$ are positive parameters.  Assume that there are $\beta>0$ and $\bm{z}$ such that $V_1(\bm{z})-V_2(\bm{z})=\beta$. We choose $\delta>0$ such that $2\delta\langle\bm{z}\rangle\leq \frac{\beta}{2}$ such that for $0<m\leq 1$,
\begin{equation}
\label{co1}
\frac{\beta}{2}<\beta-2\delta\langle \bm{z} \rangle^m=\Phi(\bm{z},\bm{z})\leq \sup \Phi(\bm{x},\bm{y}).
\end{equation}
Since $\Phi$ is continuous and $\lim_{\|\bm{x}\|+\|\bm{y}\|\rightarrow \infty}\Phi(\bm{x},\bm{y})=-\infty$, there exist $\overline{\bm{x}}$, $\overline{\bm{y}}$ such that
\begin{equation}
\label{co2}
\Phi(\overline{\bm{x}},\overline{\bm{y}})=\sup\Phi(\bm{x},\bm{y}).
\end{equation}
From the inequality $\Phi(\overline{\bm{x}},\overline{\bm{x}})+\Phi(\overline{\bm{y}},\overline{\bm{y}})\leq 2\Phi(\overline{\bm{x}},\overline{\bm{y}})$ we easily get
\begin{equation}
\label{co3}
\frac{\|\overline{\bm{x}}-\overline{\bm{y}}\|^2}{\epsilon}\leq V_1(\overline{\bm{x}})-V_1(\overline{\bm{y}})+V_2(\overline{\bm{x}})-V_2(\overline{\bm{y}}).
\end{equation}
Then the boundedness of $V_1$ and $V_2$ implies that
\begin{equation}
\label{co4}
\|\overline{\bm{x}}-\overline{\bm{y}}\|\leq c\sqrt{\epsilon}
\end{equation}
for a suitable constant $c$. By plugging \eqref{co4} into \eqref{co3} and using the Lipschitz continuity of $V_1$ and $V_2$ we get
\begin{equation}
\label{co5}
\frac{\|\overline{\bm{x}}-\overline{\bm{y}}\|}{\epsilon}\leq w \sqrt{\epsilon}
\end{equation}
for some constant $w$.

Next, define the continuously differentiable functions
\begin{equation}
\begin{split}
&\phi(\bm{x}):=V_2(\overline{\bm{y}})+\frac{\|\bm{x}-\overline{\bm{y}}\|^2}{2\epsilon}+\delta(\langle\bm{x}\rangle^m+\langle\overline{\bm{y}}\rangle^m),\\
&\psi(\bm{y}):=V_1(\overline{\bm{x}})-\frac{\|\overline{\bm{x}}-{\bm{y}}\|^2}{2\epsilon}-\delta(\langle\overline{\bm{x}}\rangle^m+\langle{\bm{y}}\rangle^m),
\end{split}
\end{equation}
and observe that $V_1-\phi$ attains its maximum at $\overline{\bm{x}}$ and $V_2-\psi$ attains its minimum at $\overline{\bm{y}}$. It is easy to compute
\begin{equation}
\begin{split}
&\frac{\partial \phi(\bm{x})}{\partial \bm{x}}\mid_{\bm{x}=\overline{\bm{x}}}=\frac{\overline{\bm{x}}-\overline{\bm{y}}}{\epsilon}+\lambda \overline{\bm{x}}, \lambda=\delta m\langle\overline{\bm{x}}\rangle^{m-2},\\
&\frac{\partial \psi(\bm{y})}{\partial \bm{y}}\mid_{\bm{y}=\overline{\bm{y}}}=\frac{\overline{\bm{x}}-\overline{\bm{y}}}{\epsilon}-\tau \overline{\bm{y}}, \tau=\delta m\langle\overline{\bm{y}}\rangle^{m-2}.
\end{split}
\end{equation}
Thus, we obtain that
\begin{equation}
\begin{split}
\min\big\{\gamma V_1(\overline{\bm{x}})-H^-(\overline{\bm{x}}),V_1(\overline{\bm{x}})-h(\overline{\bm{x}})\big\} \leq \min\big\{\gamma V_2(\overline{\bm{y}})-H^-(\overline{\bm{y}}),V_2(\overline{\bm{y}})-h(\overline{\bm{y}})\big\}.
\end{split}
\end{equation}
Further, we have that
\begin{equation}
\begin{split}
\min\big\{\gamma V_1(\overline{\bm{x}})-H^-(\overline{\bm{x}})-(\gamma V_2(\overline{\bm{y}})-H^-(\overline{\bm{y}})), V_1(\overline{\bm{x}})-h(\overline{\bm{x}})-(V_2(\overline{\bm{y}})-h(\overline{\bm{y}}))\big\}\leq 0.
\end{split}
\end{equation}
Obviously, either
\begin{equation}
\label{co7}
\begin{split}
\gamma V_1(\overline{\bm{x}})-H^-(\overline{\bm{x}})-(\gamma  V_2(\overline{\bm{y}})-H^-(\overline{\bm{y}}))\leq 0~\text{or}
\end{split}
\end{equation}
 \begin{equation}
\label{co8}
V_1(\overline{\bm{x}})-h(\overline{\bm{x}})-(V_2(\overline{\bm{y}})-h(\overline{\bm{y}}))\leq 0
\end{equation}
 holds. We will obtain a contradiction separately.

If \eqref{co7} holds,
\begin{equation}
\begin{split}
V_1(\overline{\bm{x}})-V_2(\overline{\bm{y}})\leq &\frac{1}{\gamma}(H^-(\overline{\bm{x}})-H^-(\overline{\bm{y}})
\leq \frac{1}{\gamma}(L_{\bm{f}}w \sqrt{\epsilon}+\delta m K(\langle\overline{\bm{y}}\rangle^m+\langle\overline{\bm{x}}\rangle^m+\epsilon)
\end{split}
\end{equation}
where $K=L_{\bm{f}}+\sup_{\bm{u}\in U,\bm{d}\in D}\{\|\bm{f}(\bm{0},\bm{u},\bm{d})\|\}$ and the last inequality can be obtained as follows:  \begin{equation}
\begin{split}
&H^-(\overline{\bm{x}})-H^-(\overline{\bm{y}})\\
&=\sup_{\bm{d}\in D}\inf_{\bm{u}\in U}\frac{\partial \phi(\bm{x})}{\partial \bm{x}}\mid_{\bm{x}=\overline{\bm{x}}}\cdot \bm{f}(\overline{\bm{x}},\bm{u},\bm{d})-\sup_{\bm{d}\in D}\inf_{\bm{u}\in U}\frac{\partial \psi(\bm{y})}{\partial \bm{y}}\mid_{\bm{y}=\overline{\bm{y}}}\cdot \bm{f}(\overline{\bm{y}},\bm{u},\bm{d})\\
&\leq \sup_{\bm{d}\in D}\big(\inf_{\bm{u}\in U}\frac{\partial \phi(\bm{x})}{\partial \bm{x}}\mid_{\bm{x}=\overline{\bm{x}}}\cdot \bm{f}(\overline{\bm{x}},\bm{u},\bm{d}) -\inf_{\bm{u}\in U}\frac{\partial \psi(\bm{y})}{\partial \bm{y}}\mid_{\bm{y}=\overline{\bm{y}}}\cdot \bm{f}(\overline{\bm{y}},\bm{u},\bm{d}) \big)\\
&\leq \inf_{\bm{u}\in U}\frac{\partial \phi(\bm{x})}{\partial \bm{x}}\mid_{\bm{x}=\overline{\bm{x}}}\cdot \bm{f}(\overline{\bm{x}},\bm{u},\bm{d}_1)-\inf_{\bm{u}\in U}\frac{\partial \psi(\bm{y})}{\partial \bm{y}}\mid_{\bm{y}=\overline{\bm{y}}}\cdot \bm{f}(\overline{\bm{y}},\bm{u},\bm{d}_1)+\frac{\epsilon}{2}\\
&\leq \frac{\partial \phi(\bm{x})}{\partial \bm{x}}\mid_{\bm{x}=\overline{\bm{x}}}\cdot \bm{f}(\overline{\bm{x}},\bm{u}_2,\bm{d}_1)-\frac{\partial \psi(\bm{y})}{\partial \bm{y}}\mid_{\bm{y}=\overline{\bm{y}}}\cdot \bm{f}(\overline{\bm{y}},\bm{u}_2,\bm{d}_1)+\epsilon\\
&=(\frac{\overline{\bm{x}}-\overline{\bm{y}}}{\epsilon}+\lambda \overline{\bm{x}})\cdot \bm{f}(\overline{\bm{x}},\bm{u}_2,\bm{d}_1)-(\frac{\overline{\bm{x}}-\overline{\bm{y}}}{\epsilon}-\tau \overline{\bm{y}})\cdot \bm{f}(\overline{\bm{y}},\bm{u}_2,\bm{d}_1)+\epsilon\\
&\leq \frac{\|\overline{\bm{x}}-\overline{\bm{y}}\|^2}{\epsilon}L_{\bm{f}}+\lambda \overline{\bm{x}}\cdot \bm{f}(\overline{\bm{x}},\bm{u}_2,\bm{d}_1)+\tau \overline{\bm{y}}\cdot \bm{f}(\overline{\bm{y}},\bm{u}_2,\bm{d}_1)+\epsilon \\
&\leq \frac{\|\overline{\bm{x}}-\overline{\bm{y}}\|^2}{\epsilon}L_{\bm{f}}+\lambda \overline{\bm{x}}\cdot (\bm{f}(\overline{\bm{x}},\bm{u}_2,\bm{d}_1)-\bm{f}(\bm{0},\bm{u}_2,\bm{d}_1)+\bm{f}(\bm{0},\bm{u}_2,\bm{d}_1))+\\
&~~~~~~~~~~~~~~~~~~~~~~~~~~~~~~~~~~~~~~~~~~~~\tau \overline{\bm{y}}\cdot (\bm{f}(\overline{\bm{y}},\bm{u}_2,\bm{d}_1)-\bm{f}(\bm{0},\bm{u}_2,\bm{d}_1)+\bm{f}(\bm{0},\bm{u}_2,\bm{d}_1))+\epsilon\\
&\leq \frac{\|\overline{\bm{x}}-\overline{\bm{y}}\|^2}{\epsilon}L_{\bm{f}}+\lambda L_{\bm{f}}\|\overline{\bm{x}}\|^2+\lambda \|\overline{\bm{x}}\|\|\bm{f}(\bm{0},\bm{u}_2,\bm{d}_1)\|+\tau L_{\bm{f}}\|\overline{\bm{y}}\|^2+\tau \|\overline{\bm{y}}\|\|\bm{f}(\bm{0},\bm{u}_2,\bm{d}_1)\|+\epsilon\\
&\leq \frac{\|\overline{\bm{x}}-\overline{\bm{y}}\|^2}{\epsilon}L_{\bm{f}}+\lambda K(1+\|\overline{\bm{x}}\|^2)+\tau K(1+\|\overline{\bm{y}}\|^2)+\epsilon\\
&\leq L_{\bm{f}}w \sqrt{\epsilon}+\delta m K(\langle\overline{\bm{y}}\rangle^m+\langle\overline{\bm{x}}\rangle^m)+\epsilon,
\end{split}
\end{equation}
where $\bm{d}_1$ satisfies
\begin{equation}
\begin{split}
&\sup_{\bm{d}\in D}\big(\inf_{\bm{u}\in U}\frac{\partial \phi(\bm{x})}{\partial \bm{x}}\mid_{\bm{x}=\overline{\bm{x}}}\cdot \bm{f}(\overline{\bm{x}},\bm{u},\bm{d})-\inf_{\bm{u}\in U}\frac{\partial \psi(\bm{y})}{\partial \bm{y}}\mid_{\bm{y}=\overline{\bm{y}}}\cdot \bm{f}(\overline{\bm{y}},\bm{u},\bm{d})\big)\\
&\leq \inf_{\bm{u}\in U}\frac{\partial \phi(\bm{x})}{\partial \bm{x}}\mid_{\bm{x}=\overline{\bm{x}}}\cdot \bm{f}(\overline{\bm{x}},\bm{u},\bm{d}_1)-\inf_{\bm{u}\in U}\frac{\partial \psi(\bm{y})}{\partial \bm{y}}\mid_{\bm{y}=\overline{\bm{y}}}\cdot \bm{f}(\overline{\bm{y}},\bm{u},\bm{d}_1)+\frac{\epsilon}{2}
\end{split}
\end{equation}
 and $\bm{u}_2$ satisfies 
 $$\inf_{\bm{u}\in U}\frac{\partial \psi(\bm{y})}{\partial \bm{y}}\mid_{\bm{y}=\overline{\bm{y}}}\cdot \bm{f}(\overline{\bm{y}},\bm{u},\bm{d}_1)\geq \frac{\partial \psi(\bm{y})}{\partial \bm{y}}\mid_{\bm{y}=\overline{\bm{y}}}\cdot \bm{f}(\overline{\bm{y}},\bm{u}_2,\bm{d}_1)-\frac{\epsilon}{2}.$$

 Therefore, choosing $0<m\leq \frac{\gamma}{K}$, we obtain
$$\Phi(\overline{\bm{x}},\overline{\bm{y}})\leq V_1(\overline{\bm{x}})-V_2(\overline{\bm{y}})-\delta(\langle \overline{\bm{x}}\rangle^m +\langle \overline{\bm{y}}\rangle^m)\leq \frac{1}{\gamma}(L_{\bm{f}}w \sqrt{\epsilon}+\epsilon).$$ $\Phi(\overline{\bm{x}},\overline{\bm{y}})$ can be smaller than $\frac{\beta}{2}$ for $\epsilon$ small enough, contradicting \eqref{co1} and \eqref{co2}.

If \eqref{co8} holds, $$\Phi(\overline{\bm{x}},\overline{\bm{y}})\leq V_1(\overline{\bm{x}})-V_2(\overline{\bm{y}})\leq h(\overline{\bm{x}})-h(\overline{\bm{y}})\leq L_h c\sqrt{\epsilon},$$ where $L_h$ is the Lipschitz constant over a local compact region in $\mathbb{R}^n$ covering $\overline{\bm{x}}$ and $\overline{\bm{y}}$, $\Phi(\overline{\bm{x}},\overline{\bm{y}})$ can be smaller than $\frac{\beta}{2}$ for $\epsilon$
small enough, contradicting \eqref{co1}.

 Above all, $V_1\leq V_2$ over $\bm{x}\in \mathbb{R}^n$. It is evident that if $U(\bm{x})$ is a bounded Lipschitz continuous viscosity solution to \eqref{HJB1}, then $U(\bm{x})=V^-(\bm{x})$ over $\bm{x}\in \mathbb{R}^n$, due to the fact that $U(\bm{x})$ and $V^-(\bm{x})$ are both sub and superviscosity solutions. Therefore, the uniqueness of the bounded Lipschitz continuous solutions to \eqref{HJB1} is guaranteed.
 \end{proof}

We continue exploiting more on $V^-$ and $V^+$ based on \eqref{HJB1} and \eqref{HJB2}. According to Lemma \ref{cont} stating that $V^-$ and $V^+$ are Lipschitz continuous, we have the following corollary.
  \begin{corollary}
 The Hamilton-Jacobi equation \eqref{HJB1}(\eqref{HJB2}) for $V^-$($V^+$) holds classically a.e. in $\mathbb{R}^n$, i.e. except on a set of measure $0$.
 \end{corollary}
 \begin{proof}
 By Lemma \ref{cont}, $V^-$($V^+$) is Lipschitz and, hence, by Rademacher's theorem, they are differentiable a.e., which implies that the equation \eqref{HJB1}(\eqref{HJB2}) that $V^-$($V^+$) satisfies in Theorem \ref{viscosity1} holds classically in these points.
 \end{proof}

$V^{-}(\bm{x})\leq V^{+}(\bm{x})$ always holds for $\bm{x}\in \mathbb{R}^n$. Moreover, we have that $V^-=V^+$ under the Isaacs condition.
  \begin{theorem}
  \label{pluseq}
  $V^{-}(\bm{x})\leq V^{+}(\bm{x})$ holds for $\bm{x}\in \mathbb{R}^n$ and consequently $\mathcal{R}^+\subseteq \mathcal{R}^-$. Moreover,  if $H^-(\bm{x},\bm{p})=H^+(\bm{x},\bm{p})$ for $\bm{x}\in \mathbb{R}^n$ and $\bm{p}\in \mathbb{R}^n$, then $V^-=V^+$ and thus $\mathcal{R}^-=\mathcal{R}^+$.
  \end{theorem}
  \begin{proof}
The conclusions regarding that $V^{-}(\bm{x})\leq V^{+}(\bm{x})$ holds for $\bm{x}\in \mathbb{R}^n$ and $V^-=V^+$ can be assured by Corollary 2.2 in Chapter VIII of \cite{Bardi1997}.  It is obvious that  $\mathcal{R}^+\subseteq \mathcal{R}^-$ according to the fact that $V^{-}(\bm{x})\leq V^{+}(\bm{x})$ for $\bm{x}\in \mathbb{R}^n$ and Lemma \ref{inclusion}. 

If $V^-=V^+$, we have $\{\bm{x}\mid V^-(\bm{x})=0\}=\{\bm{x}\mid V^+(\bm{x})=0\}$. According to Lemma \ref{inclusion}, we have $\mathcal{R}^-=\mathcal{R}^+$.
\end{proof}

\begin{remark}
\label{remark1}
Since $\bm{f}(\bm{x},\bm{u},\bm{d}):\mathbb{R}^n\times U \times D \mapsto \mathbb{R}^n$ is continuous over $\bm{x}$, $\bm{u}$ and $\bm{d}$, according to Theorem 2.3 in Chapter VIII of \cite{Bardi1997}, if $U$ and $D$ are convex spaces, the sets $\{\bm{u}\in U\mid H(\bm{u},\overline{\bm{d}})\geq t\}$ and $\{\bm{d}\in D\mid H(\overline{\bm{u}},\bm{d})\geq t\}$ are convex for all $t\in \mathbb{R}$, $\overline{\bm{u}}\in U$, $\overline{\bm{d}}\in D$, where $H(\bm{u},\bm{d})=\bm{p}\cdot \bm{f}(\bm{x},\bm{u},\bm{d})$. Then $H^-(\bm{x},\bm{p})=H^{+}(\bm{x},\bm{p})$. 

The simplest system for Theorem 2.3 in Chapter VIII in \cite{Bardi1997} to hold is the one being affine in the control variables $\bm{u}\in U$ and $\bm{d}\in D$, that is, 
$\bm{f}(\bm{x},\bm{u},\bm{d})=\bm{f}_1(\bm{x})+\bm{f}_2(\bm{x})\bm{u}+\bm{f}_3(\bm{x})\bm{d}$, where $U$ and $D$ are convex compact sets in $\mathbb{R}^m$ and $\mathbb{R}^l$ respectively. 
\end{remark}

\section{Examples}
\label{examples}
In this section we illustrate our approach on one example. All computations were performed on an i7-7500U 2.70GHz CPU with 4GB RAM running Ubuntu 17. For numerical implementation, we employ the ROC-HJ solver \cite{bokanowski2013} \footnote{https://uma.ensta-paristech.fr/soft/ROC-HJ/} for solving Hamilton-Jacobi equations \eqref{HJB1} and \eqref{HJB2}.
\begin{example}
\label{ex1}
\textbf{Moore-Greitzer jet engine model.} We test our approach on the following polynomial system coming from \cite{sassi2012}, corresponding to a Moore-Greitzer model of a jet engine:
\begin{equation}
\begin{split}
&\dot{x}=-y-\frac{3}{2}x^2-\frac{1}{2}x^3+d,\\
&\dot{y}=(0.8076+u)x-0.9424y,
\end{split}
\end{equation}
where $\mathcal{X}=\{\bm{x}\mid h(\bm{x})\leq 0\}$ with $h(\bm{x})=\frac{x^2+y^2-0.25}{1+(x^2+y^2-0.25)^2}$, $d\in [-0.02,0.02]$ and $u\in [-0.01,0.01]$.

From \cite{sassi2012}, we know that $u(\bm{x})=0.8076x-0.9424y$ is a controller that guarantees the existence of a robust invariant set of the following system 
\begin{equation}
\begin{split}
&\dot{x}=-y-\frac{3}{2}x^2-\frac{1}{2}x^3+d,\\
&\dot{y}=u
\end{split}
\end{equation}
where $d\in [-0.02,0.02]$.  In our example we change the coefficient $0.8076$ of the variable $x$ in $u(\bm{x})$ to $0.8076+u$ with $u\in [-0.01,0.01]$. 

The level sets of viscosity solutions  $V^-(\bm{x})$ and $V^{+}(\bm{x})$ to \eqref{HJB1} and \eqref{HJB2} are illustrated in Fig. \ref{fig1} and \ref{fig2} respectively. The visualized results in Fig. \ref{fig1} and \ref{fig2} justify  Proposition \ref{bounded} that both $V^-(\bm{x})$ and $V^{+}(\bm{x})$ are non-negative over $\bm{x}\in \mathbb{R}^n$. 

On the other side, the corresponding zero level sets $\{\bm{x}\mid V^-(\bm{x})=0\}$ and $\{\bm{x}\mid V^+(\bm{x})=0\}$ are respectively showcased in Fig. \ref{fig1} and \ref{fig2} as well. According to Lemma \ref{inclusion}, $\mathcal{R}^-=\{\bm{x}\mid V^-(\bm{x})=0\}$ and $\mathcal{R}^+=\{\bm{x}\mid V^+(\bm{x})=0\}$.  The comparison between the two robust controlled invariant  sets $\mathcal{R}^-$ and $\mathcal{R}^+$ is demonstrated in Fig. \ref{fig5}.  It is difficult to distinguish them. Actually, they are the same: According to Remark \ref{remark1} and Theorem \ref{pluseq}, $\{\bm{x}\mid V^-(\bm{x})=0\} =\{\bm{x}\mid V^+(\bm{x})=0\}$ and $\mathcal{R}^-=\mathcal{R}^+$.

\begin{figure}
\begin{minipage}{.5\textwidth}
\includegraphics[width=2.5in,height=3.0in]{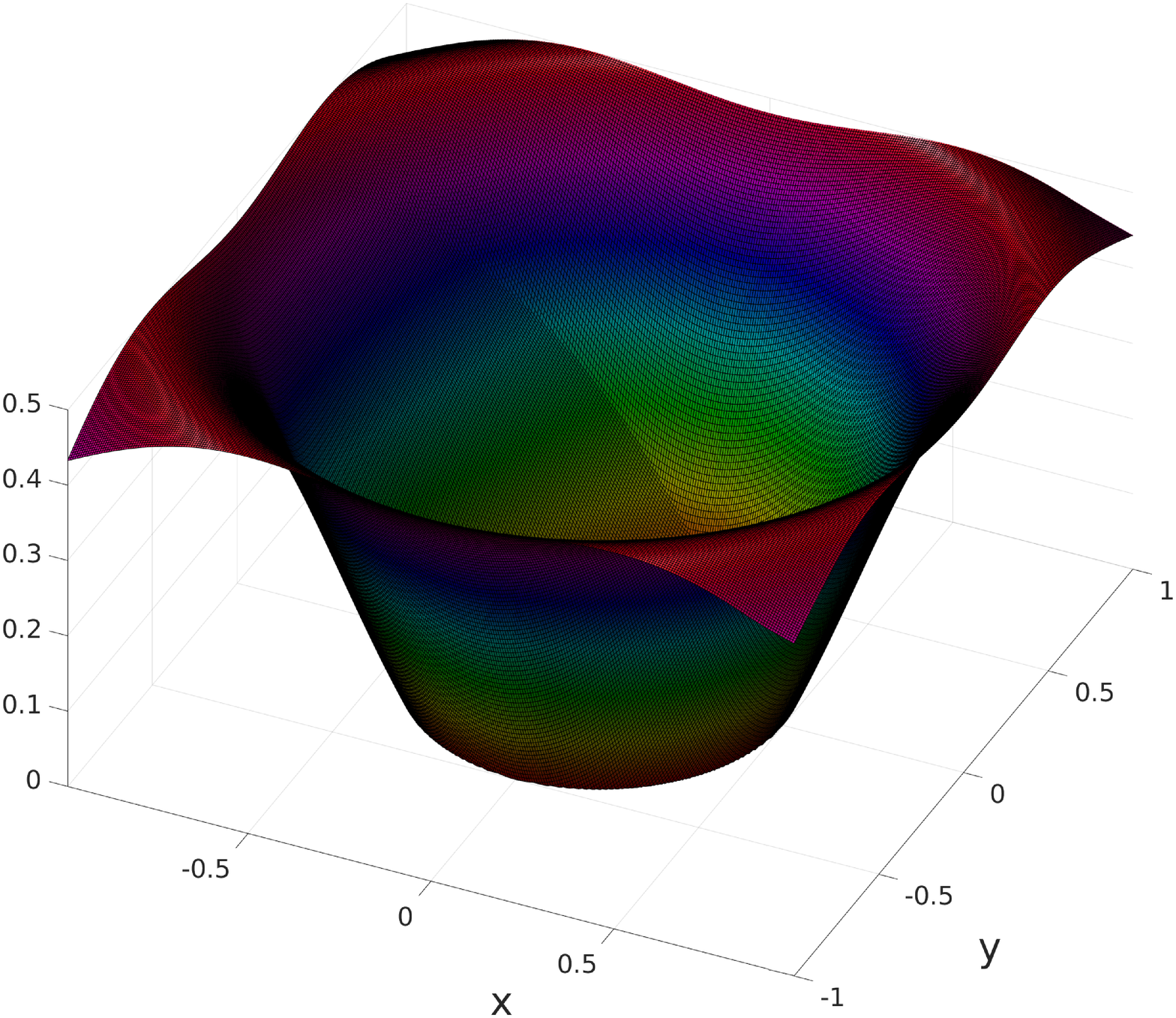} 
\end{minipage}
\begin{minipage}{.5\textwidth}
\includegraphics[width=2.5in,height=3.0in]{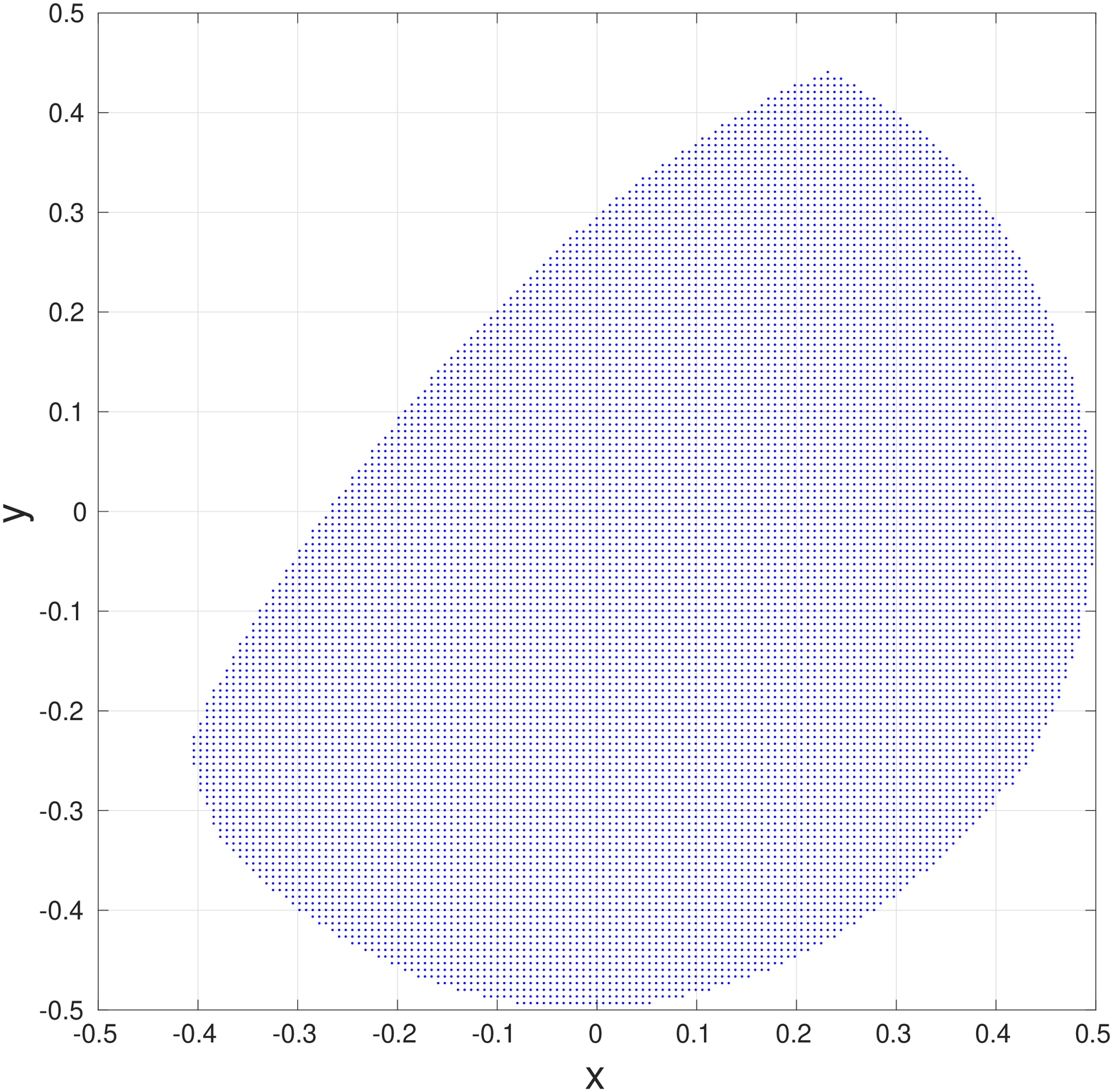}
\end{minipage}
\caption{Left: An illustration of the level sets of $V^-$ for Example \ref{ex1}. Right: An illustration of the lower robust controlled invariant  set $\mathcal{R}^-=\{\bm{x}\mid V^-(\bm{x})=0\}$(Blue region denotes $\mathcal{R}^-$).} 
\label{fig1}
\end{figure}

\begin{figure}
\begin{minipage}{.5\textwidth}
\includegraphics[width=2.5in,height=3.0in]{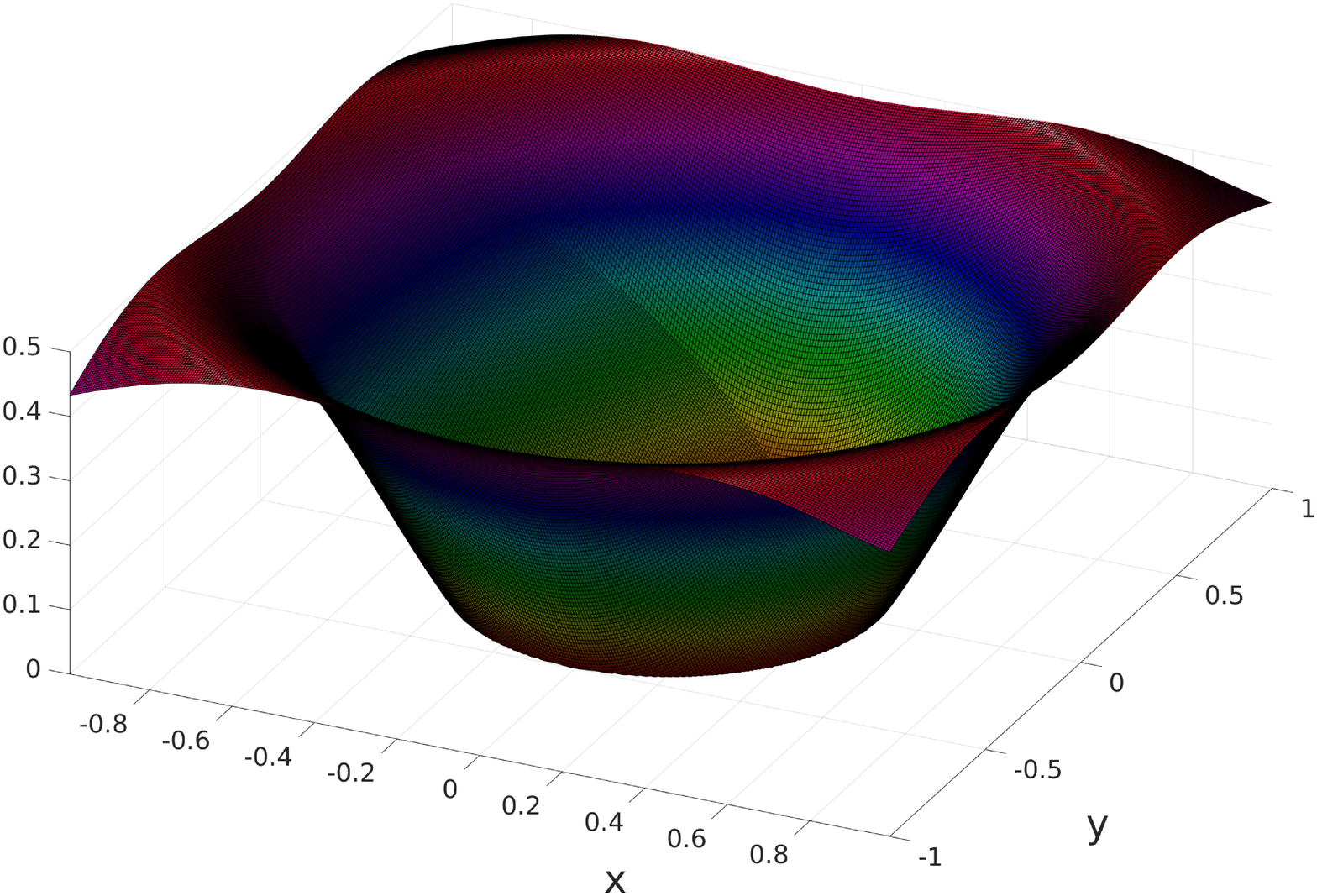} 
\end{minipage}
\begin{minipage}{.5\textwidth}
\includegraphics[width=2.5in,height=3.0in]{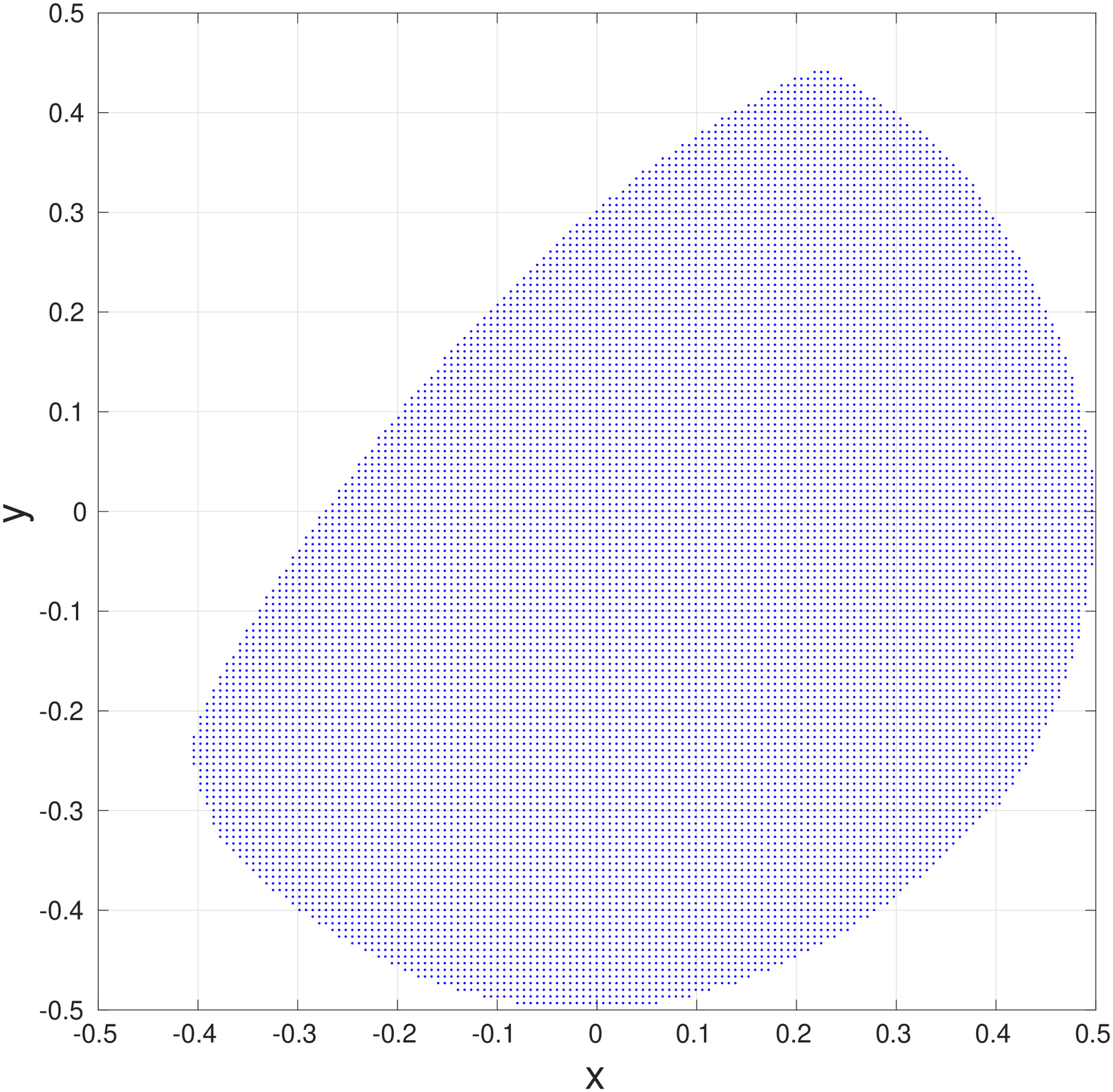} 
\end{minipage}
\caption{Left: An illustration of the level sets of $V^+$ for Example \ref{ex1}. Right: An illustration of the upper robust controlled invariant set $\mathcal{R}^+=\{\bm{x}\mid V^+(\bm{x})=0\}$ for Example \ref{ex1}(Blue region denotes $\{\bm{x}\mid V^+(\bm{x})=0\}$).} 
\label{fig2}
\end{figure}

\begin{figure}
\centering
\includegraphics[width=4.0in,height=2.5in]{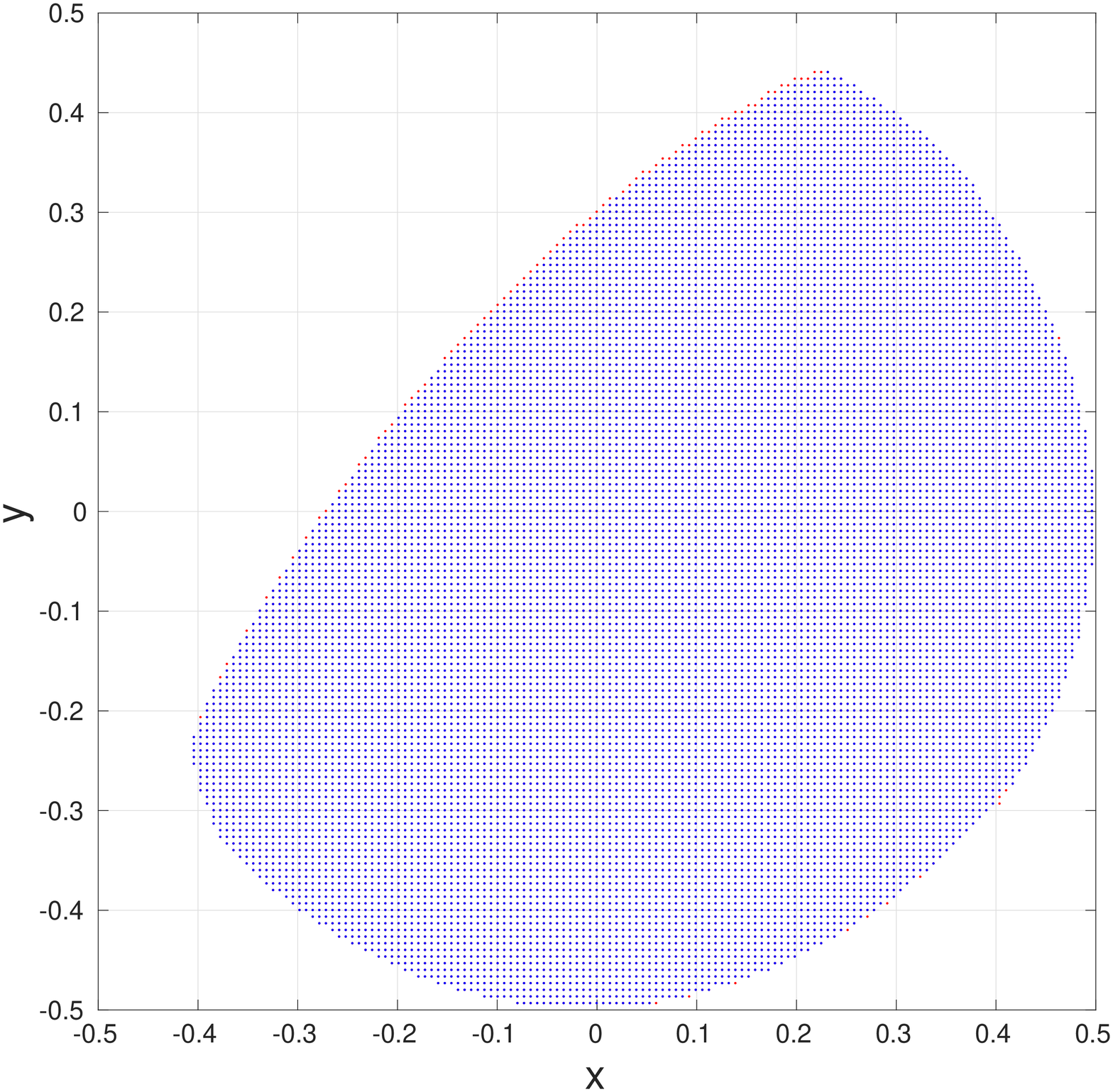} 
\caption{A comparison of the sets $\mathcal{R}^-=\{\bm{x}\mid V^-(\bm{x})=0\}$ and $\mathcal{R}^+=\{\bm{x}\mid V^+(\bm{x})=0\}$ for Example \ref{ex1}(Red region denotes $\mathcal{R}^-$. Blue region denotes $\mathcal{R}^+$).}
\label{fig5}
\end{figure}

\end{example}

\section{Conclusion and Future Work}
\label{con}
In this paper we considered infinite time reach-avoid differential game, in which the ego player aims to make the system satisfy certain safety specification perpetually while the mutual other player attempts to prevent the ego player from succeeding. This game was studied within the Hamilton-Jacobi reachability framework, in which the lower robust controlled invariant set is  the zero level set of the unique bounded Lipschitz continuous viscosity solution to a Hamilton-Jacobi equation with sup-inf Hamiltonian while the upper robust controlled invariant set is characterised as  the zero level set of the unique bounded Lipschitz continuous viscosity solution to a Hamilton-Jacobi equation with inf-sup Hamiltonian. Under the classical Isaacs condition, these two sets are equal. One example adopted from Moore-Greitzer model of a jet engine was employed to illustrate our approach.

In the future work we will further explore the relationship between the lower robust controlled invariant set $\mathcal{R}^-$ (the upper robust controlled invariant set $\mathcal{R}^+$) and the lower controlled invariant set $\mathcal{R}^{*-}$(the upper controlled invariant set $\mathcal{R}^{*+}$), where the concepts of the lower controlled invariant set $\mathcal{R}^{*-}$ and the upper controlled invariant set $\mathcal{R}^{*+}$ are given in Definition \ref{IC}. 
 \begin{definition}
 \label{IC}
 Let $\mathcal{X}=\{\bm{x}\in \mathbb{R}^n\mid h(\bm{x})\leq 0\}$ be a compact set in $\mathbb{R}^n$, where $h(\bm{x})$ is a bounded and locally Lipschitz continuous function in $\mathbb{R}^n$,

1)  The lower controlled invariant set $\mathcal{R}^{*-}$ of system \eqref{system} is the set of states $\bm{x}$'s such that there exists a non-anticipative strategy $\bm{\alpha}_{\bm{x}}(\cdot)\in \Gamma$ such that for any perturbation $\bm{d}(\cdot)\in \mathcal{D}$ the corresponding trajectory $\bm{\phi}_{\bm{x}}^{\bm{\alpha}_{\bm{x}}(\bm{d}),\bm{d}}(t)$ stays inside $\mathcal{X}$ for $t\geq 0$,
 i.e.,
\begin{equation*}
\begin{split}
 \mathcal{R}^{*-}=\{\bm{x}\in \mathbb{R}^n\mid &\exists \bm{\alpha}_{\bm{x}}(\cdot) \in \Gamma, \forall \bm{d}(\cdot) \in \mathcal{D}, \bm{\phi}_{\bm{x}}^{\bm{\alpha}_{\bm{x}}(\bm{d}),\bm{d}}(t)\in \mathcal{X} \text{ for }t\in [0,\infty)\}.
 \end{split}
 \end{equation*} 
 2). The upper controlled invariant set $\mathcal{R}^{*+}$ of system \eqref{system} is the set of states $\bm{x}$'s such that there exists a control action $\bm{u}_{\bm{x}}(\cdot)\in \mathcal{U}$ for all non-anticipative strategies $\bm{\beta}(\cdot)\in \Delta$ such that the trajectory $\bm{\phi}_{\bm{x}}^{\bm{u}_{\bm{x}},\bm{\beta}(\bm{u}_{\bm{x}})}(t)$ stays inside $\mathcal{X}$ for $t\geq 0$, i.e.,
\begin{equation*}
\begin{split}
\mathcal{R}^{*+}=\{\bm{x}\in \mathbb{R}^n\mid &\forall \bm{\beta}(\cdot) \in \Delta, \exists\bm{u}_{\bm{x}}(\cdot) \in \mathcal{U}, \bm{\phi}_{\bm{x}}^{\bm{u}_{\bm{x}},\bm{\beta}(\bm{u}_{\bm{x}})}(t)\in \mathcal{X} \text{ for } t\in [0,\infty)\}.
 \end{split}
\end{equation*} 
 \end{definition}

\bibliographystyle{abbrv}
\bibliography{references}


\end{document}